%% file: paper.tex
\documentclass{amsart} 

\input{preamble.tex}

\usepackage{hyperref}

\title{Kashaev--Reshetikhin Invariants of Links}
\author[K.-C.\@ Chen]{Kai-Chieh Chen}
\author[C.\@ McPhail-Snyder]{Calvin McPhail-Snyder}
\author[S.\@ Morrison]{Scott Morrison}
\author[N.\@ Snyder]{Noah Snyder}

\begin{document}

\begin{abstract}
  Kashaev and Reshetikhin previously described a way to define holonomy invariants of knots using quantum $\mathfrak{sl}_2$ at a root of unity.
  These are generalized quantum invariants depend both on a knot $K$ and a representation of the fundamental group of its complement into $\mathrm{SL}_2(\mathbb{C})$;
  equivalently, we can think of $\mathrm{KR}(K)$ as associating to each knot a function on (a slight generalization of) its character variety.
  In this paper we clarify some details of their construction.
  In particular, we show that for $K$ a hyperbolic knot $\mathrm{KaRe}(K)$ can be viewed as a function on the geometric component of the $A$-polynomial curve of $K$.
  We compute some examples at a third root of unity.
\end{abstract}

\maketitle

\tableofcontents

\section{Introduction}
Kashaev and Reshetikhin used the braiding \cite{MR2184023} on quantum $\mathfrak{sl}_2$ at a root of unity to construct \cite{MR2131015} invariants of tangles equipped with flat $\mathfrak{sl}_2$-connections on their complements.
We refer to their invariant as a \emph{holonomy invariant} because it depends on both topology of the tangle $T$ and the \emph{holonomy representation} $\rho : \pi_1(S^3 \setminus T) \to \mathrm{SL}_2(\mathbb{C})$ associated to the flat $\mathfrak{sl}_2$-connection.
Our goal in this paper is to clarify some of the details in their construction in order to obtain a well-defined knot invariant, describe some of its properties, and to compute some examples.

\subsection{Reshetikhin--Turaev and Kashaev--Reshetikhin invariants}

Let's begin by recalling the Reshetikhin--Turaev \cite{MR1036112} prescription for producing knot invariants from representations of quantum groups, modified in a standard way so that it applies to representations of quantum dimension $0$.

\begin{enumerate}
\item Consider a knot $K$, a quantum group $U_q(\mathfrak{g})$, and an irreducible representation $V$.
\item Fix a Morse presentation of the string knot $K'$ obtained by cutting open $K$ at a point.
\item Break up $K'$ into a sequence of cups, caps, and crossings.
\item Interpret each elementary piece as a map between tensor products of $V$ and $V^*$, using the braiding (for crossings), evaluation, coevaluation, quantum trace, and quantum cotrace (for cups and caps, depending on orientation).
\item Compose all these maps giving a map $f: V\rightarrow V$.
  Because $V$ is irreducible, $f = \langle f \rangle \id_V$ is a multiple of the identity.
The scalar $\langle f \rangle$ is an invariant of $K$.
\end{enumerate}

The last step is necessary if the quantum dimension, i.e.\@ the quantum trace of $\id_V$, is zero, as it is in our examples.
There is a more sophisticated way to deal with this problem, the \emph{modified dimensions} of \citeauthor{MR2480500} \cite{MR2480500}.
We could use this theory to extend%
\footnote{
  In this paper we are working with the category of weight \emph{bimodules} over $\funcalg$.
  Computations of modified dimensions in the literature (such as \cite{MR3143580}) mostly consider the category of weight \emph{modules} using a different braiding.
  In principle this could make a significant difference, and it would be interesting to compute the bimodule case in detail.
}
our construction to links (since it shows that the invariants do not depend on which component we cut) but since we are mostly interested in knots we do not do this here.

The Kashaev--Reshetikhin construction is a modified version of this procedure for the quantum group $\funcalg[q]$ for $q = \zeta$ a root of unity.
The algebra $\funcalg[q]$ is a slight variant of the  Kac--De Concini unrestricted quantum group with a different normalization of the generators $E$ and $F$;
they are isomorphic whenever $q - q^{-1}$ is invertible.%
\footnote{
  Our change in normalization simplifies the correspondence between central characters and group elements given in Proposition \ref{thm:spec-z0-desc}.
  More abstractly, it is so that at $q = 1$ we obtain the algebra of functions on the group $\operatorname{SL}_2^*(\mathbb C)$ instead of the universal enveloping algebra $\mathcal{U}(\mathfrak{sl}_2)$.
}

The key difference with the RT construction is the braiding.
Instead of being an isomorphism between modules $V \otimes W \rightarrow W \otimes V$, the braiding is a map of modules ${V} \otimes {W} \rightarrow {W'} \otimes {V'}$ where $W'$ and $V'$ are new modules determined by $V$ and $W$.
Furthermore this braiding is only defined for a generic choice of $V$ and $W$.
(Here generic means that their central characters live in a specific Zariski open subspace of $\mathrm{Spec}(Z(\funcalg))=\SL_2^*(\C)$.)

Thus, before we can extract an invariant following the above process, we first need to fix a labeling of the knot diagram by generic representations of $F_\zeta(\SL_2(\C))$ compatible with the action of the braiding.
Kashaev and Reshetikhin observed that such a labeling is closely related to a choice of flat $\mathfrak{sl}_2$-connection in the complement of the knot, and it is well-known that a flat connection $\alpha$ on $M$ is essentially equivalent to a representation $\pi_1(M) \to \SL_2(\C)$ up to conjugacy.

\citeauthor{MR4073970} \cite{MR4073970} made this connection precise as follows.
Let $\zeta$ be a primitive $\ell$th root of unity.
Then a generic labelling of a diagram of a link $L$ in $S^3$ by representations of $\funcalg$ is exactly the same as a generic map in $\Hom{}{\knotgrpg}{\SL_2(\C)}$.
Here $\knotgrpg = \knotgrpg(L)$ is the $\ell$th generalized knot group of $L$; an element of $\Hom{}{\knotgrpg}{\SL_2(\C)}$ is roughly the same thing as a homorphism $\pi(S^3 \setminus L) \to \SL_2(\C)$ along with a choice of $\ell$th root%
\footnote{
  More accurately, it is an $\ell$th root when $\ell$ is odd.
  When $\ell$ is even it is instead a $(\ell/2)$th root and there is an extra sign.
  See Definition \ref{def:knotgrpg} for details.
}
of each meridian.
This can be formalized as a \emph{generic biquandle factorization} of the \emph{conjugation quandle} associated with $\SL_2(\C)$.

In summary, the Kashaev-Reshetikhin invariant ($\mathrm{KaRe}(K)$ for short) is a rational function on the affine variety $\Hom{}{\knotgrpg}{\mathrm{SL}_2(\mathbb{C})}$ which can be produced by the following recipe:

\begin{enumerate}
  \item Fix a Morse presentation of the $(1,1)$-tangle (string knot) $K'$ obtained by cutting open $K$ at a point.
\item Use the representation of the generalized knot group and the factorization map to assign to each edge an element of $\SL_2^*(\C)$ together with a choice of $r$th root for its eigenvalues.
  (In some cases we might have to conjugate the representation $\rho$ first, but this can always be done.)
\item Use the inverse of the central character map to give a labeling of the edges by irreducible representations of the quantized ring of functions $\funcalg$.
\item Replace the irrep $V$ by the simple bimodule $\End{V}$.
\item Use the formulas for braiding, cup, and cap to interpret $K'$ as a map of simple $\funcalg$-bimodules.
  The generic condition assures that the braiding is well-defined.
\item Since the bimodule is simple, this gives a number, the value of the Kashaev-Reshetikhin invariant for that point in  $\Hom{}{\knotgrpg}{\mathrm{SL}_2(\mathbb{C})}$.
\end{enumerate}
As with the RT construction, this process gives a functor from a category of $G$-graded tangle diagrams to a braided monoidal category.
It is a special case of the functors defined in \cite{MR2131015}.
The modification using $(1,1)$-tangle diagrams can be understood in terms of the \emph{quasi-functors} of \cite[Section 4]{MR2480500}.

\subsection{Geometric interpretation}
The generalized representation variety $\repvarg{K} = \Hom{}{\knotgrpg}{\mathrm{SL}_2(\mathbb{C})}$ is a complicated space (it is a finite-sheeted cover of the representation variety of $K$) so our construction produces a rather complicated invariant.
By looking more closely at the structure of $\repvarg{K}$ we can simplify it.
Because $\kare(K)$ is constant on the orbits of the conjugation action of $\SL_2(\C)$, we can produce a rational function on the (extended) character variety, i.e.\@ the GIT quotient $\charvarg{K} = \repvarg{K}//\SL_2(\C)$.

The character variety $\charvar K$ (and its extended version $\charvarg K$) is still somewhat complicated.
One way to give it a simpler description is to use the peripheral subgroup of $\pi_1$ generated by the longitude and meridian.
By considering the eigenvalues $L^{\pm 1}$ and $M^{\pm 1}$ of these two elements we get a subvariety of $\C P^2$ cut out by the \emph{$A$-polynomial} \cite{MR1628754,MR1288467}, a Laurent polynomial $\apoly K(M,L)$.
When our knot $K$ is hyperbolic $\charvar K$ has a distinguished \emph{geometric component}, and it turns out \cite[Theorem 3.1]{MR1695208} that this component is birationally equivalent to a component $\apolycurve K$ of the $A$-polynomial curve.

The same story works for the generalized character variety $\charvarg K$ and a simple generalization $\apolyg K(\mu, L)$ of the $A$-polynomial; here the variable $\mu$ is an $\ell$th root%
\footnote{As before, the situation is slightly more complicated when $\ell$ is even.}
of the variable $M$ of the $A$-polynomial.
In particular, we can think of our invariant $\kare$ as defining a rational function on the curve $\apolycurveg K$ cut out by the generalized $A$-polynomial 
of $M$.
This gives a simpler invariant of $K$ that still incorporates geometric information.

When $K$ is hyperbolic it (by definition) has a representation $\rho_{\text{hyp}} \in \repvar K$ corresponding to the complete finite-volume hyperbolic structure on the complement of $K$.
This representation is unique up to conjugation, so it gives a distinguished point of $\charvar K$.
The geometric component of $\charvar K$ is the component containing $\rho_{\text{hyp}}$.
The value $\kare(K, \rho_{\text{hyp}})$ of $\kare$ at this point is a scalar invariant of $K$ related to the hyperbolic geometry of $S^3 \setminus K$.

\subsection{Relationship with other work}
This article is based on an unfinished preprint by SM and NS (2010) and the PhD thesis \cite{KCCThesis} of KCC (2019).
The final version of the article was prepared by CMS.

In parallel to this work, \citeauthor{MR4073970} gave \cite{MR4073970} a construction of holonomy invariants closely related to the original \cite{MR2131015} Kashaev--Reshetikhin invariants, including the details of the generic correspondence between points of $\Hom{}{\knotgrpg}{\SL_2(\C)}$ and labeled diagrams.
They used this to describe a holonomy invariant of links denoted $\operatorname{BGPR}$.
CMS showed \cite{2005.01133} that for $\zeta = i$ the quantum double (a sort of norm-square) of $\mathrm{BGPR}(L, \rho)$ is equal to the Reidemeister torsion of $S^3 \setminus L$ twisted by $\rho$.

The BGPR construction uses simple modules $V$, while our invariant is defined using simple bimodules $\End{V}$.
From this, one might expect (up to normalization)
\[
  \operatorname{BGPR}(K,\rho)
  \operatorname{BGPR}(\overline{K},\rho)
  =
  \kare(K,\rho)
\]
where $\overline{K}$ is the mirror image of $K$.
However, this equation does not appear to hold, at least in this simple form.
It would be interesting to better understand the relationship between $\operatorname{BGPR}$ and $\kare$.

\section*{Acknowledgements}
The authors would all like to thank Nicolai Reshetikhin for introducing them to holonomy invariants and for many helpful suggestions throughout the project.
We also thank Ian Agol for some clarifying remarks about geometric structures on knot complements.
NS was supported by NSF DMS grant number 2000093.
SM and NS would like to thank Microsoft Station Q for hospitality during a visit which started their collaboration on this paper.

\section{Quantum \texorpdfstring{$\mathfrak{sl}_2$}{sl2} at a root of unity}
\subsection{The quantized function algebra}
\begin{defn}
  Let $\SL_2(\C)$ be the usual special linear group, i.e.\@ the group of $2 \times 2$ complex matrices with determinant $1$.
  The \emph{Poisson dual group}\footnote{$\SL_2(\C)$ is a Poisson--Lie group, so $\mathfrak{sl}_2$ is a Poisson--Lie bialgebra. Taking the dual of this structure and integrating to a group gives $\SL_2^*(\C)$.} of $\SL_2(\C)$ is the group
  \[
    \SL_2^*(\C) \defeq \left\{
      \left(
      \begin{pmatrix}
        \kappa & 0 \\
        \phi & 1
      \end{pmatrix},
      \begin{pmatrix}
        1 & \epsilon \\
        0 & \kappa
      \end{pmatrix}
    \right)
    \, \middle | \,
    \kappa \ne 0
  \right\}
  \subseteq
  \mathrm{GL}_2(\C) \times \mathrm{GL}_2(\C).
  \]
\end{defn}
The quantized algebra of functions on this group is the key ingredient in our construction.

\begin{defn}
  Let $\funcalg[q]$ be the algebra over $\Q[q,q^{-1}]$ generated by $E$, $F$, and $K^{\pm 1}$, subject to the following relations:
  \begin{gather*}
    KE=q^2 EK, \\
    KF=q^{-2}FK\text{, and} \\
    [E,F]=(q-q^{-1})(K-K^{-1}).
  \end{gather*}

This algebra is a Hopf algebra with the coproduct defined by:
$$\Delta(K)=K\otimes K,$$
$$\Delta(E)=E\otimes K + 1 \otimes E\text{, and}$$
$$\Delta(F)=F\otimes 1 + K^{-1} \otimes F.$$

The counit and antipode are given by, $\varepsilon(K)=1$, $\varepsilon(E)=\varepsilon(F)=0$, $S(K)=K^{-1}$, $S(E)=EK^{-1}$, and $S(F)=KF$.
\end{defn}
This algebra is almost identical to the unrestricted quantum group of De Concini and Kac.
Its main difference is the specialization at $q=1$, where the De Concini--Kac form has been modified to yield $U(\mathfrak{sl}_2)$, while our algebra specializes to the ring of functions $\funcalg[]$.

This algebra can be specialized to any point of $\mathrm{Spec}(\Q[q,q^{-1}])$; in particular we can specialize $q$ to any nonzero complex number.
We care about the specializations of the form $\funcalg$ for $\zeta$ a primitive $\ell$th root of unity.
From now on, we take $\zeta$ to be such a root for odd%
\footnote{
  When $\ell = 2r$ is even, we instead use $r$th powers of the generators and there are some extra signs.
  Both even and odd $\ell$ are considered in \cite[Section 6]{MR4073970}, and \cite{McPhailSnyderThesis} focuses on the even case.
}
$\ell \ge 3$.

The key fact about $\funcalg$ is that is has a large center:
\begin{defn}
  Let $\Fcent$ be the subalgebra of  $\funcalg$ generated by $E^\ell, F^\ell, K^{\pm \ell}$.
\end{defn}

\begin{prop}
  \label{thm:spec-z0-desc}
  $\Fcent$ is a central Hopf subalgebra.
  Futhermore, $\spec{\Fcent}$ is isomorphic as an algebraic group to $\SL_2^*(\C)$, via the following recipie for thinking of elements of $\SL_2^*(\C)$ as functions on $\Fcent$:
  \begin{equation}
    \label{eq:char-to-mat}
      \left(
      \begin{pmatrix}
        \kappa & 0 \\
        \phi & 1
      \end{pmatrix},
      \begin{pmatrix}
        1 & \epsilon \\
        0 & \kappa
      \end{pmatrix}
    \right)
    \leftrightarrow
    \begin{aligned}
      E^\ell &\mapsto \epsilon \\
      F^\ell & - \phi/\kappa \\
      K^\ell &\mapsto \kappa
    \end{aligned}
  \end{equation}
\end{prop}
We will identify (closed, complex) points of $\SL_2^*(\C)$ with $\Fcent$-characters, that is homomorphisms $\chi : \Fcent \to \C$.
In this case the above recipie associates $\chi$ to the group element
\[
      \left(
      \begin{pmatrix}
        \chi(K^\ell) & 0 \\
        - \chi(K^\ell F^\ell) & 1
      \end{pmatrix},
      \begin{pmatrix}
        1 & \chi(E^\ell) \\
        0 & \chi(K^\ell)
      \end{pmatrix}
    \right) \in \SL_2^*(\C).
\]

\begin{defn}
  The \emph{Casimir element} of $\funcalg$ is
  \[
    \Omega = EF + \zeta^{-1} K + \zeta K^{-1}.
  \]
\end{defn}

\begin{thm}
  \label{thm:casimir-relation}
  The center of $\funcalg$ is generated by $\Fcent$ and $\Omega$.
  The only polynomial relation between the generators is
  \begin{align*}
    E^\ell F^\ell &= \prod_{m = 0}^{\ell -1} (\Omega - \zeta^k K - \zeta^{-k} K^{-1})
    \intertext{equivalently}
    \cb_\ell(\Omega) &= K^\ell + K^{-\ell} + E^r F^r
  \end{align*}
  where $\cb_\ell$ is the $\ell$th renormalized Chebyshev polynomial, determined by $\cb_\ell(2 \cos \theta) = 2 \cos(\ell \theta)$.
\end{thm}
Essentially, this says that the spectrum of the center of $\funcalg$ is a $\ell$-fold cover of $\spec \Fcent = \SL_2^*(\C)$, so irreducible representations of $\funcalg$ are parametrized by points of $\SL_2^*(\C)$ along with a choice of $r$th root.
This explains why we use representations of the generalized knot group instead of the usual knot group.

\subsection{The Casimir and eigenvalues}
\begin{defn}
  Let $\chi$ be a $\Fcent$-character, equivalently a point of $\SL_2^*(\C)$.
  We write
  \begin{align*}
    \psi(\chi) &=
    \begin{pmatrix}
      \chi(K^\ell) & 0 \\
      - \chi(K^\ell F^\ell) & 1
    \end{pmatrix}
    \begin{pmatrix}
      0 & \chi(E^\ell) \\
      0 & \chi(K^\ell)
    \end{pmatrix}^{-1} \\
    &=
    \begin{pmatrix}
      \chi(K^\ell) & - \chi(E^\ell) \\
      - \chi(K^\ell F^\ell) & \chi(K^{-\ell} + E^\ell F^\ell)
    \end{pmatrix}
  \end{align*}
  for the \emph{defactorization} of $\chi$.
  By abuse of notation, we set
  \begin{align*}
    \tr \chi = \tr \psi (\chi) = \chi(K^\ell + K^{-\ell} + E^\ell F^\ell)
  \end{align*}
\end{defn}
Later in Section \ref{sec:colored-tangle-diagrams} we will associate a strand of a knot diagram labeled by $\chi$ with (a conjugate of) the matrix $\psi(\chi)$.
In particular, we are interested in the eigenvalues of $\psi(\chi)$.
The choice of an $r$th root of these eigenvalues is closely related to the action of the Casimir:
\begin{prop}
  Let $V$ be an irreducible $\funcalg$-module.
  Then $\Fcent$ acts on $V$ by some character $\chi$.
  Furthermore, if $\tr \chi = \lambda + \lambda^{-1}$, then $\Omega$ acts on $V$ by $\mu + \mu^{-1}$, where $\mu^\ell = \lambda$.
\end{prop}
\begin{proof}
  The first claim is Schur's lemma.
  For the second, suppose that $\Omega$ acts on $V$ by $\omega = \zeta^{\alpha} + \zeta^{-\alpha} = 2 \cos(\pi \alpha /\ell)$ for some $\alpha \in \C$.
  Then by Theorem \ref{thm:casimir-relation},
  \begin{align*}
    \tr \chi
    = \chi(\cb_\ell(\Omega))
    = \cb_r(2 \cos (\pi \alpha /\ell))
    = 2 \cos(\pi \alpha)
    = \zeta^{\ell \alpha} + \zeta^{- \ell \alpha}
  \end{align*}
  and we see that $\mu = \zeta^{\alpha}$, $\lambda = \zeta^{\ell \alpha}$.
\end{proof}

\section{Representations of the quantum group}
\label{sec:qgrp-reps}
As with any algebra, representations of $\funcalg$ are parametrized by the spectrum of its center.
As seen in the previous section, the center of $\funcalg$ is an $r$-fold cover of the algebraic group $\SL_2^*(\C)$.
We want to restrict to a certain family of well-behaved modules parametrized by this group.

\begin{defn}
  Let $V$ be a  $\funcalg$-module.
  We say it is a \emph{weight module} if the center $\Fcent[\Omega]$ of $\funcalg$ acts diagonalizably on $V$.
  We write $\modcat$ for the category of finite-dimensional $\funcalg$ weight modules, and $\modcat_\chi$ for the subcategory of modules for which $\Fcent$ acts by the character $\chi \in \SL_2^*(\C)$.
\end{defn}

\begin{defn}
  Let $\chi$ be a $\Fcent$-character, and let $\mu$ satisfy $\tr \chi = (\mu^\ell + \mu^{-\ell})$; in this case we call $\mu$ a \emph{fractional eigenvalue} for $\chi$.
  Write $\beta^r = \chi(K^\ell)$ and $\epsilon = \chi(E^\ell)$.
  We denote by $\irrep{\chi, \mu}$ the $\funcalg$-module with basis $v_0, \dots, v_{\ell-1}$ and action
  \begin{align*}
    K \cdot v_k &= \beta \zeta^{2k} v_k & \Omega \cdot v_k &= (\mu + \mu^{-1}) v_k & E \cdot v_k &=
    \begin{cases}
      v_{k+1} & 0 \le k < \ell-1, \\
      \epsilon v_0 & k = \ell-1
    \end{cases}
  \end{align*}
  This presentation depends on a choice of $\ell$th root $\beta$ of $\chi(K^\ell)$, but this choice does not affect the isomorphism class of $\irrep{\chi, \mu}$ so we usually ignore it.
  Similarly $\irrep{\chi, \mu} \iso \irrep{\chi, \mu^{-1}}$.
\end{defn}
The representations $\irrep{\chi, \mu}$ are isomorphic to the \emph{cyclic representations} of \cite[Section 4.1]{KCCThesis}, although we have given a slightly different description of them.
They are also the same as the \emph{cyclic modules} of \cite[Section 6]{MR4073970}, although we include the cases $\tr \chi = \pm 2$.

\begin{prop}
  The representations $\irrep{\chi, \mu}$ are irreducible $\ell$-dimensional weight modules for $\funcalg$.
\end{prop}

Given a $\Fcent$-character $\chi$, a choice of fractional eigenvalue $\mu$ extends $\chi$ to a character $\hat \chi : \Fcent[\Omega] \to \C$ by $\hat \chi(\Omega) = \mu + \mu^{-1}$.
\begin{prop}
  The quotient $\funcalg/\ker \hat \chi$ is a simple $\funcalg$-bimodule of dimension $\ell^2$ and $\End{\irrep{\chi, \mu}} \iso \funcalg/\ker \hat \chi$.
\end{prop}
\begin{proof}
  This follows from the fact that $\irrep{\chi, \mu}$ is simple and the center of $\funcalg$ is $\Fcent[\Omega]$.
  To see that $\End{\irrep{\chi, \mu}}$ has dimension $\ell^2$, observe that by Theorem \ref{thm:casimir-relation} it has a basis $K^i E^j$ for $0 \le i,j \le \ell-1$.
\end{proof}
We introduce the bimodules $\End{\irrep{\chi, \mu}}$ in order to unambiguously define the braiding, as discussed in the next section.

\section{Braiding on the quantized function algebra}
Unlike $\funcalg[q]$, the algebra $\funcalg$ is not quasitriangular.%
\footnote{
  Technically speaking $\funcalg[q]$ is not quasitiangular either: only the $h$-adic version $\funcalg[h]$ is.
  When $q$ is generic the action of the formal power series $R$-matrix converges on every weight module, but when $q = \zeta$ is a root of unity this is no longer the case.
}
Instead, there is an outer automorphism
\[
  \rmat : \funcalg \otimes \funcalg \to (\funcalg \otimes \funcalg)[W^{-1}],
\]
where $W = 1 + (-1)^{\ell}K^{-r} E^r \otimes F^r K^r$,
satisfying the Yang-Baxter equations
\begin{align*}
  (\Delta \otimes 1) \rmat(u \otimes v) &= \rmat_{13} \rmat_{23} (\Delta(u) \otimes v) \\
  (1 \otimes \Delta) \rmat(u \otimes v) &= \rmat_{13} \rmat_{12} (u \otimes \Delta(v))
\end{align*}
and
\begin{align*}
  (\epsilon \otimes 1) \rmat(u \otimes v) &= \epsilon(u) v \\
  (1 \otimes \epsilon) \rmat(u \otimes v) &= \epsilon(v) u .
\end{align*}
Here $\Delta$ is the coproduct, $\epsilon$ the counit, and $\rmat_{ij}$ means the action on the $i$th and $j$th tensor factors.

In a quasitriangular Hopf algebra, $\rmat$ comes from conjugation by an element called the \emph{$R$-matrix}.
This is not the case for $\funcalg$, but there is a version of $\funcalg$ defined over formal power series in $h$ (with $q = e^h $) which has an $R$-matrix.
The conjugation action of this element is still well-defined in the specialization $q = \zeta$, giving the outer automorphism $\rmat$.
For more details, see \cite{MR2184023}.

We can characterize the action of $\rmat$ on $\funcalg \otimes \funcalg$ by
\begin{align}
  \label{eq:rmat-rel-i}
  \rmat(E \otimes 1) &= E \otimes K \\
  \rmat(1 \otimes F) &= K^{-1} \otimes F \\
  \rmat(1 \otimes K) &= (1 \otimes K) X^{-1} \\
  \label{eq:rmat-rel-iv}
  \rmat(\Delta(u)) &= \Delta^{\mathrm{op}}(u), \quad u \in \funcalg.
\end{align}
where $X = 1 - \zeta K^{-1}E \otimes F K$.

\begin{prop}
$\rmat$ extends uniquely to an algebra automorphism $\funcalg \otimes \funcalg \to (\funcalg \otimes \funcalg)[W^{-1}]$ and 
\[
  \rmat(\Omega \otimes 1) = \Omega \otimes 1, \quad
  \rmat(1 \otimes \Omega ) = 1 \otimes \Omega.
\]
\end{prop}

Now let $\chi_1, \chi_2$ be $\Fcent$-characters with fractional eigenvalues $\mu_1, \mu_2$.
If $(\chi_1 \otimes \chi_2)(W) \ne 0$, then the automorphism $\rmat$ descends to a map of bimodules
\[
  \rmat : \End{\irrep{\chi_1, \mu_1}} \otimes \End{\irrep{\chi_2, \mu_2}} \to 
  \End{\irrep{\chi_3, \mu_1}} \otimes \End{\irrep{\chi_4, \mu_2}}
\]
where $\chi_3, \chi_4$ are in general different from both $\chi_1$ and $\chi_2$, determined by the rule
\[
  (\chi_3 \otimes \chi_4)\rmat = \chi_1 \otimes \chi_2. 
\]
This corresponds to the fact that the holonomy of a path around a strand of a knot diagram changes at the crossings.
We can describe $\chi_3$ and $\chi_4$ in terms of $\chi_1$ and $\chi_2$ by computing the action of $\rmat$ on $\Fcent \otimes \Fcent$, but there is a more enlightenting geometric interpretation given in Section \ref{sec:colored-tangle-diagrams}.

As mentioned above, the braiding $\rmat$ is only defined when
\[
  (\chi_1 \otimes \chi_2)(W) = 1 - \chi_1 (K^{-\ell} E^\ell) \chi_2(F^\ell K^\ell) \ne 0.
\]
When this holds, we say that the pair $(\chi_1, \chi_2)$ is \emph{admissible}.
Because not all pairs of $\Fcent$-characters are admissible, the braiding is only partially defined.
Dealing with this rigorously requires introducing the concept of a \emph{generically defined biquandle} \cite[Section 5]{MR4073970}.
The key observation is that the set of admissible pairs is large (Zariski open and dense), so we can always globally conjugate away from them.
As such the partially-defined braiding is rarely an issue in practice.

One advantage of defining a braiding using the map $\rmat$ is that its matrix consists of rational functions in the parameters of the characters.
\begin{prop}
  With respect to the bases%
  \footnote{
    Technically this is a basis only when  $\chi_i(E^\ell) \ne 0$, so $E$ acts invertibly.
    Generically this is true; when it is not, we can instead use the basis $ K^i F^j$, or either in the case where $\chi_i(E^\ell) = \chi_i(F^\ell) = 0$.
  }
  $\{K^i E^j : 0 \le i, j \le \ell-1\}$ of $\End{\irrep{\chi_i, \mu_i}}$ the matrix components of 
  \[
  \rmat : \End{\irrep{\chi_1, \mu_1}} \otimes \End{\irrep{\chi_2, \mu_2}} \to 
  \End{\irrep{\chi_3, \mu_1}} \otimes \End{\irrep{\chi_4, \mu_2}}
  \]
  are rational functions in the parameters
  \[
    \chi_i(K^\ell), \chi_i(E^\ell), \chi_i(F^\ell), \chi_i(\Omega) = (\mu + \mu^{-1}).
  \]
\end{prop}
\begin{proof}
  This follows from the defining equations (\ref{eq:rmat-rel-i}--\ref{eq:rmat-rel-iv}) and using the relation
  \[
    F = (\Omega - \zeta K - \zeta^{-1} K^{-1}) E^{-1}
  \]
  with $E^{-1} = \chi(E^r)^{-1} E^{r-1}$ to eliminate powers of $F$.
  To see that terms with $X^{-1}$ do not cause any problems, observe that
  when $X = 1 - \zeta K^{-1} E \otimes FK \in \End{\irrep{\chi_3, \mu_3}} \otimes \End{\irrep{\chi_4, \mu_4}}$ we have
  \[
    X^{-1} =
    \frac{\chi_3(K^\ell)}{ \chi_3(K^\ell) - \chi_3(E^\ell) \chi_4(K^\ell) \chi_4(F^\ell)}
    \sum_{m = 0}^{\ell - 1} (\zeta K^{-1} E \otimes KF)^j.\qedhere
  \]
\end{proof}

Notice that, since $\rmat$ is an \emph{outer} automorphism, it does not automatically descend to an $R$-matrix on modules, by which we mean a map
\[
  R_{V, W} : V \otimes W \to V' \otimes W'.
\]
In the usual Reshetkhin--Turaev construction, the action of the universal $R$-matrix $R \in \funcalg[q] \otimes \funcalg[q]$ on $V \otimes W$ converges defines such an $R$-matrix $R_{V,W}$.
However, when $E \otimes F$ does not act nilpotently on $V \otimes W$, the action of $R$ no longer converges.
When our representation $\rho \in \charvarg K$ is irreducible $E \otimes F$ is generically invertible so we cannot directly%
\footnote{
  There are workarounds, such as in \cite[Theorem 6.2]{MR4073970} but these introduce phase ambiguities.
  It seems that eliminating them requires extra structure on the generalized character variety $\charvarg K$.
}
define $R_{V,W}$ using the universal $R$-matrix.
Avoiding this problem is our motivation for replacing the simple module $V$ with the simple bimodule $\End{V}$.

\section{Colored tangle diagrams}
\label{sec:colored-tangle-diagrams}

In this section we describe how to describe a knot $K$ with a point 
\[
  \rho \in \repvarg K = \Hom{}{\knotgrpg(K)}{\SL_2(\C)}
\]
of its generalized representation variety in a form convenient for $\funcalg$.
Our treatment is somewhat informal; for all the details and a more general theory, see \cite{MR4073970}.

Let $K$ be a knot in $S^3$.
Recall that the \emph{knot group} $\pi(K)$ of $K$ is the fundamental group of the complement of $K$ in $S^3$.
The knot group has distinguished elements called \emph{meridians}; roughly speaking, a meridian is a path that wraps perpendicularly around a strand of $K$.
There are many meridians, but they are all conjugate.
(For a link, instead all meridians of each component are conjugate.)

It is possible to define a \emph{generalized knot group} $\pi^{(\ell)}(K)$ which contains an $\ell$th roots of the meridians of $K$ \cite{MR1167178, KellyThesis}.
Instead of defining $\pi^{(\ell)}(K)$ in full detail, we simply describe what homomorphisms from it to $\SL_2(\C)$ are.
Specifically, we show how to extend $\rho \in \Hom{}{\pi(K)}{\SL_2(\C)}$ to an element of $\Hom{}{\knotgrpg{K}}{\SL_2(\C)}$.

The idea is as follows: pick a meridian $\mathfrak m \in \pi(K)$.
We can diagonalize\footnote{If $\rho(\mathfrak m)$ has $\pm 1$ as an eigenvalue, then we instead put $\mathfrak m$ in Jordan form. Since we don't actually need the entries of $\rho(\mathfrak{m}^{1/\ell})$ this does not cause any problems.} $\rho(\mathfrak{m})$ as
\[
  \rho(\mathfrak{m}) =
  \begin{pmatrix}
    M  \\
    & M^{-1}
  \end{pmatrix},
\]
and then choosing an $\ell$th root means choosing $\mu$ with $\mu^\ell = M$ :
\[
  \rho(\mathfrak{m}^{1/\ell}) =
  \begin{pmatrix}
    \mu \\
    & \mu^{-1}
  \end{pmatrix}.
\]
Since the other meridians are conjugate to $\mathfrak{m}$, we can compute them in terms of the above choice.
\begin{defn}
  \label{def:knotgrpg}
  Let $K$ be a knot in $S^3$.
  An element of $\Hom{}{\knotgrpg K}{\SL_2(\C)}$ is a group homomorphism $\rho : \pi(K) = \pi_1(S^3 \setminus K) \to \SL_2(\C)$ along with a complex number $\mu$ such that
  \[
    \mu^{\ell} + \mu^{-\ell} = \tr \rho(\mathfrak m)
  \]
  where $\mathfrak m$ is a meridian of $K$.
  (We could equivalently say that $\mu^\ell$ is an eigenvalue of $\rho(\mathfrak m)$.)
  We call the data $(\rho, \mu)$ a \emph{generalized $\SL_2(\C)$-representation} of $\pi(K)$ or simply a \emph{generalized representation.}
\end{defn}
The group homomorphism $\rho$ corresponds (generically) to a choice of $\Fcent$-character on each strand of a diagram of $K$, as we will describe shortly.
The choice of $\mu$ determines the action of the Casimir $\Omega$ on the bimodule assigned to each strand.

We can now explain how to associate an element of $\Hom{}{\knotgrpg K}{\SL_2(\C)}$ to a tangle diagram labelled by pairs $(\chi, \mu)$, equivalently by simple $\funcalg$-bimodules.
In this context it is most natural to use a slightly nonstandard description of the fundamental group of a knot complement.
\begin{figure}
  \centering
  \subcaptionbox{A positive crossing, where the characters satisfy $\chi_1 \otimes \chi_2 = (\chi_3 \otimes \chi_4) \rmat$.\label{fig:crossing-rule-positive}}{ 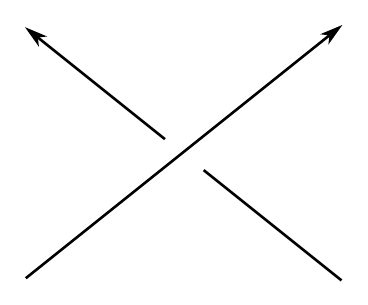 }%
  \hfill
  \subcaptionbox{A negative crossing, where the characters satisfy $\chi_1 \otimes \chi_2 = (\chi_3 \otimes \chi_4) \rmat^{-1}$.\label{fig:crossing-rule-negative}}{ 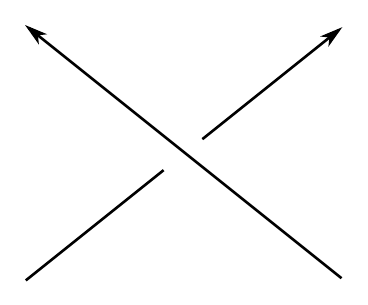 }%
  \caption{Rules for characters at crossings. Note that part of the rule is requiring that the partially-defined map $\rmat$ makes sense (i.e.\@ that $(\chi_1, \chi_2)$ is admissible).}
  \label{fig:crossing-rules} 
\end{figure}
\begin{defn}
  Let $D$ be an oriented tangle diagram.
  We can think of $D$ as a decorated $4$-valent planar graph, and we call the edges of this graph \emph{segments}.
  We say that $D$ is \emph{colored} if each segment is decorated with 
  \begin{enumerate}
    \item a central character $\chi : \Fcent \to \C$ and
    \item a fractional eigenvalue $\mu$ for $\chi$.
  \end{enumerate}
  (Equivalently, this data is a character $\hat \chi : \Fcent[\Omega]$ via  $\hat\chi(\Omega) = \mu + \mu^{-1}$.)
  The characters must satisfy compatibility conditions at each crossing of the diagram, which are given for characters in Figure \ref{fig:crossing-rules}.
  For the eigenvalues we simply require that strands $1$ and $3$ have the same value of $\mu$, and the same for strands $2$ and $4$.
\end{defn}
The condition on the fractional eigenvalues is the same as requiring that each segment of each connected component of the tangle has the same fractional eigenvalue.
Because we focus on knots, our tangles only have one connected component, so we just need to make a single global choice of $\mu$.

An example of a colored diagram is given in Figure \ref{fig:figure-eight}, which shows a $(1,1)$-tangle  whose closure is the figure-eight knot $\figeight{}$.
To avoid cluttering the picture we have not labeled every segment of the diagram; the unlabeled segments are determined completely by the rules of Figure \ref{fig:crossing-rules} once we choose the values of $x_0$ and $y_0$.

\begin{figure}
  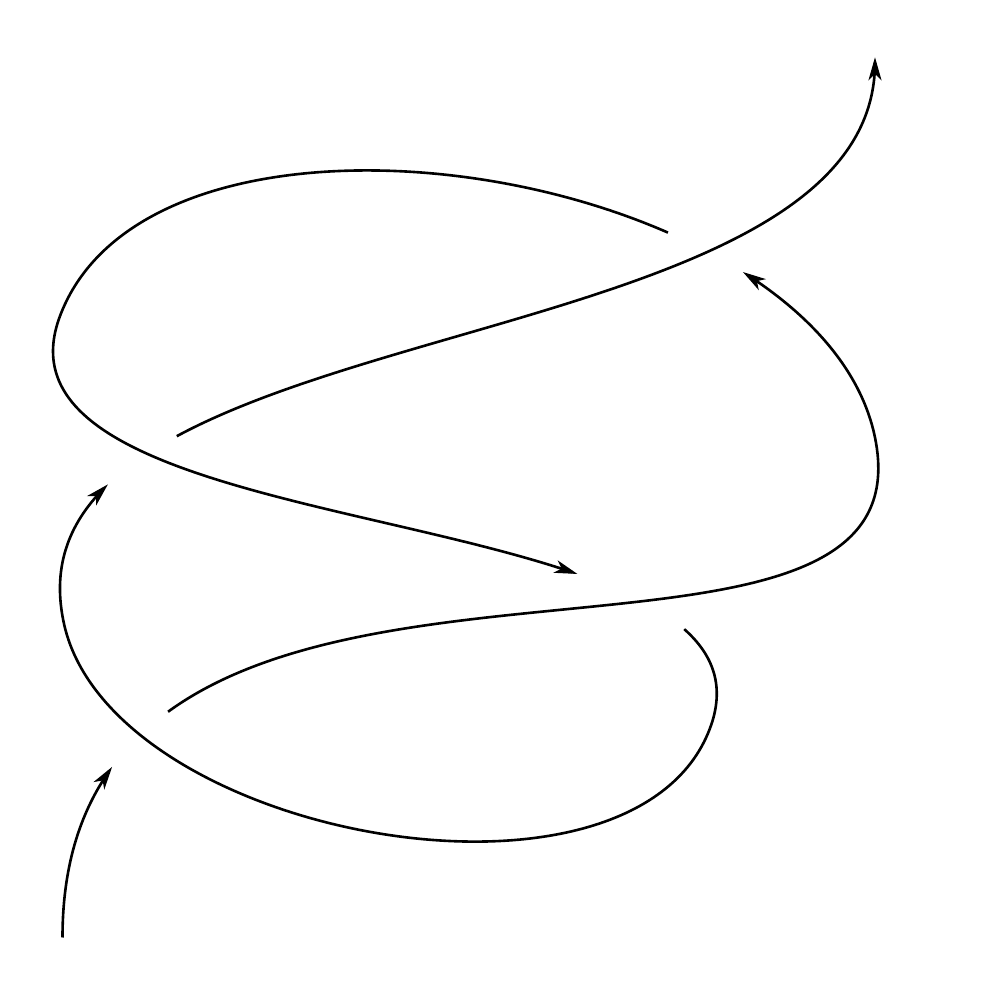 
  \caption{A $(1,1)$-tangle whose closure is the figure-eight knot. Some segments are colored by characters.
  Notice that the closure is well-defined exactly when $x_3 = x_0$ and that in general $x_1 \ne y_0$.}
  \label{fig:figure-eight}
\end{figure}

A colored tangle diagram $D$ describes a tangle plus a representation $\knotgrpg(T) \to \SL_2(\C)$ of the tangle complement.
Here by $\knotgrpg(T)$ we mean the obvious generalization of $\knotgrpg(L)$ to tangle complements: viewing a tangle as embedded in $[0,1]^3$, $\knotgrp(T)$ is the fundamental group of the complement of $T$.
Meridians still make sense in this context, and a representation $\knotgrpg(T)$ is a representation of $\knotgrp(T)$ plus fractional eigenvalues of the meridian, as before.
To describe this representation in terms of the coloring of the diagram, recall that we associated characters $\chi$ to pairs of matrices via the rule \eqref{eq:char-to-mat}.
\begin{defn}
  For any $\Fcent$-character $\chi$, the \emph{upper} and \emph{lower} holonomy are the matrices
  \[
    \chi^+ =
    \begin{pmatrix}
      \kappa & 0 \\
      \phi & 1
    \end{pmatrix}
    \text{ and }
    \chi^-
    \begin{pmatrix}
      1 & \epsilon \\
      0 & \kappa
    \end{pmatrix},
  \]
  where
  \[
    \kappa = \chi(K^\ell), \quad
    \epsilon = \chi(E^\ell), \quad \text{and } 
    \phi = - \chi(K^\ell F^\ell).
  \]  
\end{defn}

Given a path $\gamma$ in the complement of a tangle diagram we can always homotope $\gamma$ so that it only crosses segments of the diagram transversely.
At each segment $\gamma$ either goes above or below the segment.

\begin{defn}
  \label{def:holonomy-of-diagram}
  Suppose that the path $\gamma$ in the complement of a colored tangle diagram crosses segments labeled by characters $\chi_1, \dots, \chi_n$.
  The \emph{holonomy} of $\gamma$ is the matrix
  \[
    \rho(\gamma) = (\chi_1^{\eta})^{\varepsilon_1} \cdots (\chi^{\eta})^{\varepsilon_n}
  \]
  where
  \begin{itemize}
    \item $\eta_i$ is $+$ if $\gamma$ passes above the strand labeled by $\chi_i$ and $-$ if it passes  below it, and
    \item $\varepsilon_i = +1$ if $\gamma$ passes an upward-pointing strand left-to-right and $-1$ otherwise.
  \end{itemize}
\end{defn}
\begin{figure}
  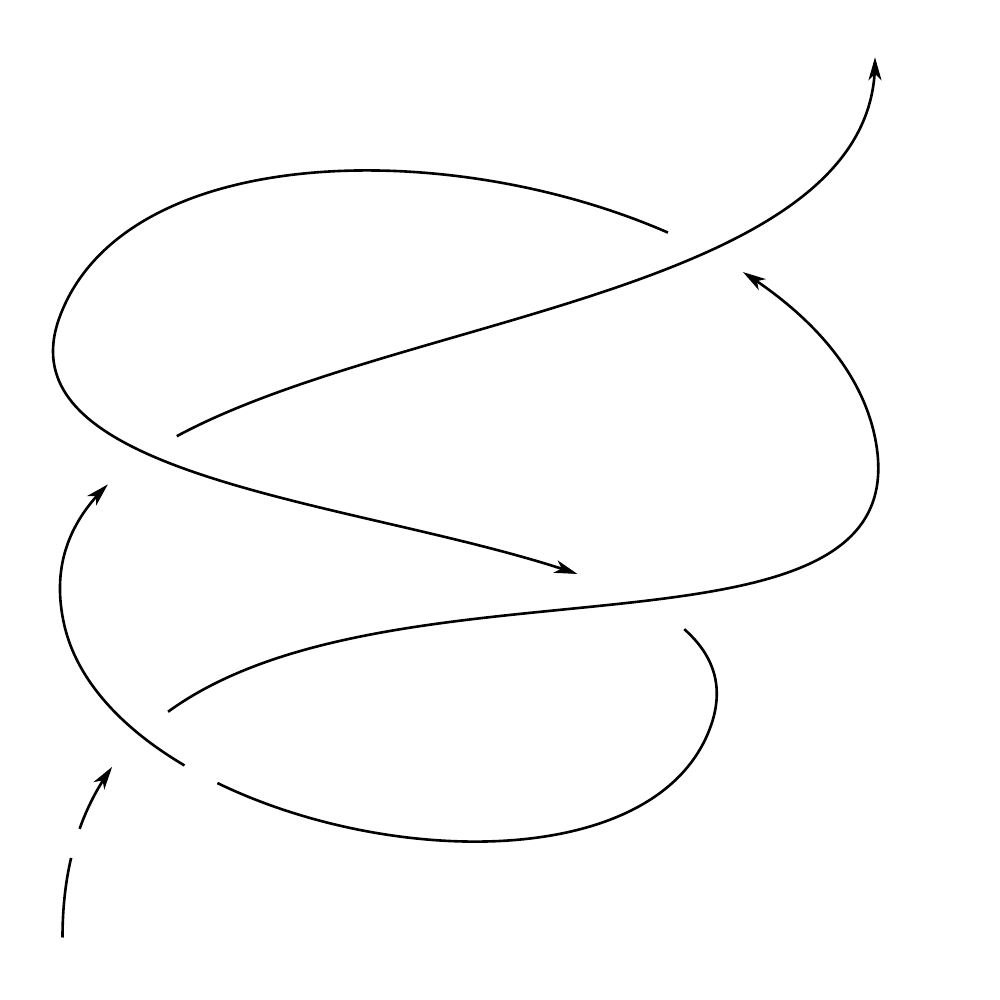 
\caption{A path in the complement with holonomy $x_0^+ y_0^+ (y_0^-)^{-1}(x_0^+)^{-1}$.}
  \label{fig:figure-eight-path-examples}
\end{figure}
For example, the path in Figure \ref{fig:figure-eight-path-examples} crosses above $x_0$ positively, above $y_0$ positively, below $y_0$ negatively, then above $x_0$ negatively, so it has holonomy
\[
  x_0^+ y_0^+ (y_0^-)^{-1}(x_0^+)^{-1}.
\]
If we close up the free ends of the diagram to obtain a diagram of the $\figeight$ knot, this path is a Wirtinger  generator (hence a meridian) of the knot group of $\figeight$.
In general, a Wirtinger generator $\mathfrak m$ corresponding to a segment colored by a character $\chi$ will have holonomy of the form
\[
  (y_1^+ \cdots y_k^+)\chi^+ (\chi^-)^{-1}(y_1^+ \cdots y_k^+)^{-1}
\]
where the characters $y_i$ represent a path from the basepoint to a region adjacent to the chosen segment.
In particular, the holonomy of $\mathfrak m$ is conjugate to the defactorization 
\[
  \psi(\chi)
  =
  \chi^+ (\chi^-)^{-1} 
  =
  \begin{pmatrix}
    \chi(K^\ell) & - \chi(E^\ell) \\
    - \chi(K^\ell F^\ell) & \chi(K^{-\ell} - (-1)^\ell E^\ell F^\ell)
  \end{pmatrix}
\]
of $\chi$.

Conversely, suppose the basepoint of the knot group is located near a segment whose Writinger generator has holonomy $g$.
To determine the character $\chi$ coloring the segment we need to factorize $g$ into upper and lower-triangular parts $g = \chi^+ (\chi^-)^{-1}$.
If the basepoint is located somewhere else, then we need to conjugate by some word in the $\chi_i^+$.

Because of this issue with the basepoints the Wirtinger generators of the knot group are somewhat difficult to describe directly.
As such, it is better to think in terms of representations of the fundamental \emph{groupoid} of the diagram complement.
This groupoid has objects regions of the diagram.
For each pair of adjacent regions there are two morphisms representing paths above and below the segment separating them, and these generate the groupoid.
(A closely related description is given in \cite[Lemma 3.4]{MR3143580}.)
One can derive the relations
\begin{align*}
  \chi_1^+ \chi_2^+ &= \chi_4^+ \chi_3^+
  \\
  \chi_1^- \chi_2^+ &= \chi_4^+ \chi_3^-
  \\
  \chi_1^- \chi_2^- &= \chi_4^- \chi_3^-
\end{align*}
between the paths (hence the holonomies) at a positive crossing by checking paths above, between, and below the strands.
For example, two paths giving the middle relation are shown in Figure \ref{fig:groupoid-relations}.
It is not hard to determine similar relations for negative crossings.

\begin{figure}
  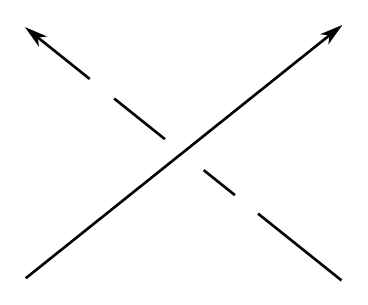
  \caption{Deriving the relation $\chi_1^- \chi_2^+ = \chi_4^+ \chi_3^-$.}
  \label{fig:groupoid-relations}
\end{figure}

\begin{lem}
  \begin{enumerate}
    \item The holonomy representation of a colored diagram is well-defined.
    \item Colorings of diagrams are compatible with Reidemeister moves.
  \end{enumerate}
\end{lem}
\begin{proof}[Proof sketch]
  (1)
  It suffices to check the relations on the holonomies $\chi_i^{\pm}$ from the fundamental groupoid of the diagram are compatible with the relations (\ref{eq:rmat-rel-i}--\ref{eq:rmat-rel-iv}) coming from $\rmat$.
  This is straightforward algebra.

  (2)
  Fundamentally this holds because $\rmat$ is invertible (RII move) and satisfies the Yang-Baxter/braid relation (RIII move).
  Formally, we say that the mapping $(\chi_1, \chi_2) \to (\chi_4, \chi_3)$ is a \emph{generically defined biquandle}, as in \cite[Definition 5.1]{MR4073970}.
\end{proof}

\begin{cor}
  Let $D$ be a diagram of a tangle $T$.
  The holonomy representation of $D$ gives a well-defined representation of the generalized knot group of $T$.
\end{cor}

In particular, this corollary says that the function $D \mapsto (K, \rho)$ taking a colored diagram to a knot plus a generalized representation of its complement is well-defined.
However, it is not surjective:
if we fix a diagram $D$ of a knot $K$, not every representation $\rho \in \Hom{}{\knotgrpg}{\SL_2(\C)}$ will correspond to a valid coloring of the diagram, because the map $\rmat$ is only partially defined on the space of characters $(\chi_1, \chi_2)$, and because not every element of $\SL_2(\C)$ lies in the image of the holonomy map.
(For example, any coloring of the diagram in Figure \ref{fig:figure-eight} assigns a path wrapping once around the segment labeled $x_0$ a matrix with nonzero $1,1$ entry.)
However, we can always conjugate $\rho$ so that it can be described by a colored diagram.

\begin{thm}
  Let $L$ be a link and $D$ a diagram of $L$.
  Any $\rho \in \Hom{}{\knotgrpg(L)}{\SL_2(\C)}$  is conjugate to some $\rho' = g \rho g^{-1}$ that can be expressed as the holonomy of a coloring of $D$.
\end{thm}
\begin{proof}
  See \cite[Section 5]{MR4073970}, in particular Theorem 5.10.
  The idea is that inadmissible colorings are rare, so their complement is a Zariski open dense set, so they can always be avoided.
\end{proof}

\section{Construction of the invariant}
Once we have a colored diagram $D$ representing a knot $K$ and a representation $\rho$ of $\knotgrpg(K)$ we can compute the complex number $\kare(K, \rho)$ by following a version of the Reshetikhin--Turaev construction.

As mentioned in the introduction we need to make a slight variation.
The modules $\End{\irrep{\chi, \mu}}$ have vanishing quantum dimension, so the value of the RT construction on any closed diagram will be $0$.
To fix this, we need to cut open the closed diagram $D$ along a strand to obtain a $(1,1)$-tangle diagram $D'$ whose closure is $D$.
For example, Figure \ref{fig:figure-eight} shows such a such a cut-open diagram of the figure-eight knot.
We will show later that the value of the invariant does not depend on where we cut open the knot.

By isotoping the diagram $D'$ we can assume it is in Morse position.%
\footnote{
  This means that the crossings of the diagrams and critical points of the edges all occur in uniform vertical layers, as in as in Figure \ref{fig:figure-eight-morse}.
}
We then use the following recipe to produce a  morphism of $\funcalg$-bimodules:
\begin{itemize}
  \item Vertical upward-pointing strands colored by $(\chi, \mu)$ are assigned the simple bimodule $\End{\irrep{\chi, \mu}}$; vertical downward-pointing strands instead get the dual module.
  \item Positive crossings (Figure \ref{fig:crossing-rule-positive}) are assigned the map $c = \tau R$ acting on the modules associated to the incoming (bottom) strands (here $\tau(v \otimes w) = w \otimes v$).
    Negative crossings (Figure \ref{fig:crossing-rule-negative}) are instead assigned  the action of $c^{-1} = \rmat^{-1} \tau$.
  \item Cups and caps get evaluation and coevaluation morphisms.
    As usual, some of these are twisted by a pivotal element $K^{r - 1}$.
    The right set of conventions \cite[equation (28)]{MR4073970} is to set
    \begin{align*}
      &\overleftarrow{\operatorname{coev}} : \C \to V \otimes V^*
      & 
      & 1 \mapsto \sum v_j \otimes v^j
      \\
      &\overleftarrow{\operatorname{ev}} : V^* \otimes V \to \C
      & 
      & f \otimes w = f(w)
      \\
      &\overrightarrow{\operatorname{ev}} : V \otimes V^* \to \C
      & 
      & v \otimes f \mapsto f(K^{1-\ell} v)
      \\
      &\overrightarrow{\operatorname{coev}} : \C \to V^* \otimes V
      & 
      & 1 \mapsto \sum v^j \otimes K^{\ell-1} v_j
    \end{align*}
    for each module $V$ with basis $\{v_j\}$ and dual basis $\{v^j\}$.
\end{itemize}
Notice that the pivotal element is $K^{\ell -1}$, not $K^{-1}$ as for generic $q$.
This is worked out in \cite[Appendix A]{MR1881401}; in particular, compare Propositions A.5 and A.13.

For example, we assign the Morse tangle diagram in Figure \ref{fig:figure-eight-morse} the map
\begin{equation*}
  (\overleftarrow{\operatorname{ev}} \otimes \id)
  (\id \otimes c)
  (c \otimes \id)
  (\id \otimes c^{-1})
  (c \otimes \id)
  (\id \otimes \overleftarrow{\operatorname{coev}})
  :
  \End{\irrep{\chi_0, \mu}}
  \to
  \End{\irrep{\chi_0, \mu }}
\end{equation*}
acting on modules $\End{\irrep{\chi, \mu}}$ and their duals%
\footnote{
  To avoid working with dual modules we can choose to work with braid diagrams and their closures, so that dual modules only show up in terms of (partial) quantum traces.
}
determined by the coloring of the diagram.
The color $\chi_0$ is the one assigned to the free strands at the top and the bottom.
\begin{figure}
  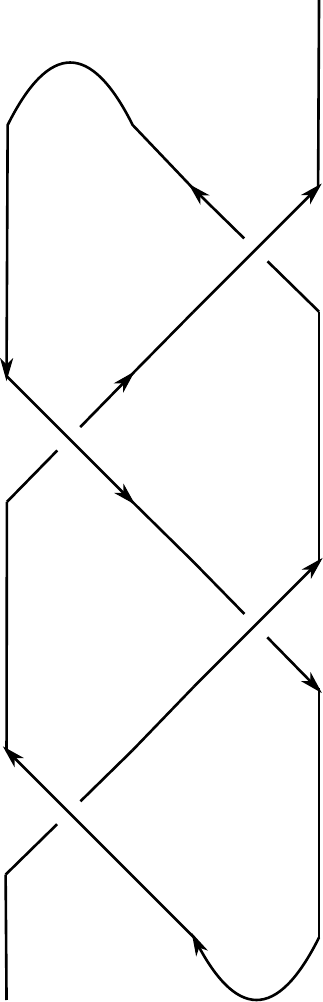
  \caption{A Morse presentation of the tangle diagram in Figure \ref{fig:figure-eight}.
  }
  \label{fig:figure-eight-morse}
\end{figure}
In general, for any diagram $D'$ whose free strands are colored by $(\chi, \mu)$ we obtain a map
\[
  \kare(D') : \End{\irrep{\chi,\mu}} \to \End{\irrep{\chi,\mu}}.
\]
Because $\End{\irrep{\chi,\mu}}$ is a simple bimodule, any such map is a scalar multiple of the identity:
\[
  \kare(D') =
  \left\langle \kare(D') \right\rangle \id_{\End{\irrep{\chi,\mu}}}.
\]
We then define
\[
  \kare(K, \rho) = \left\langle \kare(D') \right\rangle \in \C.
\]
\begin{lem}
  \begin{enumerate}
    \item The number $\kare(D')$ does not depend on the choice of where to cut the diagram open, so it is an invariant of $D$.
    \item The number $\kare(D)$ does not depend on the choice of diagram $D$ used to represent $(K, \rho)$.
  \end{enumerate}
\end{lem}
\begin{proof}
  Claim (2) follows from (1) and the usual Reshetikhin--Turaev construction.
  (1) is a special case of \cite[Lemma 2]{MR2480500}.
  In that article, \citeauthor{MR2480500} show how to extend our cutting-open construction to an invariant of links by using the values on certain open Hopf links containing an \emph{ambidextrous object} $W$.
  In our we can choose $W = \End{\irrep{\epsilon, 1}}$, where $\epsilon$ is the identity element of $\SL_2^*(\C)$, which is to say that
  \[
    \epsilon(K^\ell) = 1, \epsilon(E^\ell) = \epsilon(F^\ell) = 0.
  \]
  It can be shown that, in the notation of \cite[Lemma 2]{MR2480500},
  \[
    S'(W,U) = S'(W,V) \text{ and } S'(U,W) = S'(V,W)
  \]
  whenever the simple objects $U$ and $V$ have holonomy lying in the same gauge class, for example by considering a diagram like Figure \ref{fig:gauge-transf-example} with $D$ an open Hopf link.
  For a knot the modules coloring any strand of any diagram lie in the same gauge class (their meridians are conjugate), so we can ignore the normalization factors $S'$.
\end{proof}
It follows that $\kare(K,\rho)$ is well-defined as an invariant of $K$ and $\rho \in \repvarg K$.

\begin{figure}
  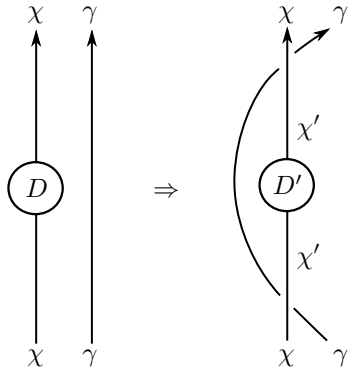
  \caption{The diagram $D' : \chi' \to \chi'$ is a gauge transformation of (i.e.\@ its holonomy is conjugate to) the diagram $D : \chi \to \chi$.}
  \label{fig:gauge-transf-example}
\end{figure}

\begin{lem}
  $\kare(K,\rho)$ does not depend on the gauge class of $\rho$: if $\rho'$ is another representation with $\rho' = g \rho g^{-1}$ for $g \in \SL_2(\C)$, then $\kare(K, \rho) = \kare(K, \rho')$.
\end{lem}
\begin{proof}
  This is one of the main results of \citeauthor{MR4073970}, specifically \cite[Theorem 5.11]{MR4073970}.
  The key idea of the proof is to view gauge transformations $\rho \mapsto g \rho g^{-1}$ in terms of applying Reidemeister moves to the diagrams, as in Figure \ref{fig:gauge-transf-example}.
  Since the functor defining $\kare$ respects Reidemeister moves it also respects gauge transformations.

  A key hypothesis is that the modified dimension function is gauge-invariant.
  In our case, we are choosing the naive modified dimension $\mathbf{d}(\End{\irrep{\chi, \mu}}) = 1$ for every $\chi$ and $\mu$.
  Because we are only working with knots, \emph{every} strand is colored by something in the gauge class of $(\chi, \mu)$ (which is in fact determined by $\mu$) so this modified dimension function is gauge-invariant.
\end{proof}
  As mentioned in the introduction, extending our invariant to links would require a more sophisticated computation of the modified dimensions $\mathbf{d}$.

\section{The representation variety, the character variety, and the \texorpdfstring{$A$}{A}-polynomial}

One perspective on the invariant $\kare$ is that it takes as input a knot $K$ and a generalized representation $\rho \in \Hom{}{\knotgrpg K}{\SL_2(\C)}$ and produces a complex number $\kare(K, \rho)$.
However, it is more informative to think of the \emph{family} of numbers $\kare(K, \rho)$ as $\rho$ ranges over the space of representations.
Because $\kare(K, \rho)$ depends only on the conjugacy class of $\rho$, we can instead think of $\kare(K)$ as a function on the \emph{generalized character variety} $\charvarg K$.

In fact, we can simplify further: the rather complicated space $\charvarg K$ is summarized by a simple generalization $\apolyg K$ of the $A$-polynomial of $K$, and we can think of $\kare (K)$ as a rational function on (some components of) the algebraic curve cut out by $\apolyg K$.

\begin{defn}
  The \emph{representation variety} of a knot $K$ is the space $\repvar K = \Hom{}{\pi(K)}{\SL_2(\C)}$ of homomorphisms from the knot group to $\SL_2(\C)$.
  $\repvar K$ is an affine variety, with coordinates given by the matrix elements of the images of the generators.
  The \emph{character variety} is the GIT quotient $\charvar K = \repvar K // \SL_2(\C)$, where $\SL_2(\C)$ acts on $\repvar K$ by conjugation.
\end{defn}

Explicitly, this means that the quotient map $\repvar K \to \charvar K$ is given by the trace.
Points of $\repvar K$ are representations of $\pi(K)$ into $\SL_2(\C)$, while points of $\charvar K$ are, roughly%
\footnote{
  Not every representation comes from a point of $\charvar K$ because of reducible, non-decomposable representations.
  However, it is true that every conjugacy class of completely reducible (hence abelian) representations and of irreducible representations comes from a point of $\charvar K$.
  These are the ones we usuallly care about, especially the irreducible representations.
}
speaking, conjugacy classes of such representations.
The Kashaev--Reshe\-tikhin invariant can be thought of a function $\charvar K \to \C$, except that we need slightly more data.

\begin{defn}
  Let $K$ be a knot, and let $\mathfrak m$ be a meridian of $K$.
  The \emph{$\ell$th generalized representation variety} $\repvarg K$ is the set of pairs $(\rho, \mu) \in \repvar K \times \C^\times$ with $\tr_\mathfrak m(\rho) = \tr \rho(\mathfrak m) = (\mu^{\ell} + \mu^{-\ell})$, i.e.\@ the fiber product
  \[
    \begin{tikzcd}
      \repvarg K \arrow[r] \arrow[d] & \C^\times \arrow[d, "f"]\\
      \repvar K \arrow[r, "\tr_{\mathfrak m}"] & \C
    \end{tikzcd}
  \]
  where $f(\mu) = (\mu^{\ell} + \mu^{-\ell})$.
  Similarly, we define the \emph{$\ell$th generalized character variety} to be the fiber product
  \[
    \begin{tikzcd}
      \charvarg K \arrow[r] \arrow[d] & \C^\times \arrow[d, "f"]\\
      \charvar K \arrow[r, "\tr_{\mathfrak m}"] & \C
    \end{tikzcd}
  \]
\end{defn}
Because every meridian of $K$ is conjugate to ${\mathfrak m}$ this definition does not depend on the choice of ${\mathfrak m}$.

\begin{thm}
  For any knot $K$, $\kare(K)$ is a well-defined function on the generalized character variety $\charvarg K$.
\end{thm}
\begin{proof}
  We showed that $\kare(K)$ is a function on the representation variety $\repvarg K$ that is invariant under conjugation.
  Each point  of $\charvarg K$ corresponds to a conjugacy class of representations, so $\kare(K)$ is well-defined on points of $\charvarg K$ as well.
\end{proof}

Explicitly describing the variety $\charvar K$ in terms of matrix coefficents of representations is somewhat difficult.
It is simpler to instead consider the curve cut out by a certain Laurent polynomial $\apoly K (M,L)$ associated to the knot $K$, the \emph{$A$-polynomial} \cite{MR1288467,MR1628754}.
Since $\kare$ is really a function on the \emph{generalized} character variety $\charvarg K$ we will need to generalize the $A$-polynomial as well.

We first recall the usual $A$-polynomial, discussed in more detail in \cite{MR1288467}.
Let $K$ be a knot in $S^3$.
We can think of the boundary of the complement $S^3 \setminus K$ as a torus, which gives an embedding $H = \Z^2 \to \pi(K)$ of the torus fundamental group into the knot group.
A choice of meridian $\mathfrak m$ and longitude $\mathfrak l$ gives a choice of generators of $H$.

Let $\rho \in \charvar K$ be a representation.
We can always conjugate in $\SL_2(\C)$ so that
\[
  \rho(\mathfrak m) =
  \begin{pmatrix}
    M & 1 \\
    0 & M^{-1}
  \end{pmatrix}
  ,
  \quad
  \rho(\mathfrak l) =
  \begin{pmatrix}
    L & q \\
    0 & L^{-1}
  \end{pmatrix}
\]
which determines a map $f : \charvar K \to \C^\times \times \C^\times$ via $f(\rho) = (M,L)$.
If we take the closure of the image of $f$, we obtain a union of $0$ and $1$-dimensional algebraic subvarieties.
The union of the one-dimensional components of the closure $\overline{f(\charvar K)}$ is a collection of algebraic curves in $\C^\times \times \C^\times$ cut out by the \emph{$A$-polynomial} $\apoly K(M,L)$.
It can be shown that $A$ is an integer polynomial defined up to multiplication by terms of the form $\pm M^a L^b$.

Now that we understand the geometric definition of the $A$-polynomial, it is easy to include the role of the fractional eigenvalues $\mu$.
\begin{defn}
  The \emph{$\ell$th generalized $A$-polynomial} of a knot $K$ is the polynomial $\apolyg K(\mu, L)$ given by replacing $M \to \mu^\ell$.
  We write $\apolycurveg K$ for the curve cut out by the polynomial $\apolyg K(\mu, L)$.
\end{defn}

We might expect that the restriction map $\charvar K \to \apolycurve K$ induced by the inclusion of the boundary will lose information, but it turns out that for certain interesting components of $\charvar K$ the map $f$ is a birational isomorphism, and similarly for $\charvarg K$.
In particular, it is an isomorphism for the commutative component (the simplest) and the geometric component, which we discuss in turn.

Let $K$ be a knot.
Any representation $\alpha \in \repvar K$ with abelian image is conjugate to one sending every meridian to a diagonal matrix
\[
  \begin{pmatrix}
    M & 0 \\
    0 & M^{-1}
  \end{pmatrix}
\]
for some $M \in \C$, which completely determines $\alpha$.
In this case the image of the longitude $\mathfrak l$ under $\alpha$ is always the identity matrix, so these representations correspond to a factor $(L - 1)$ of the $A$-polynomial of $K$.
For the generalized character variety such representations are determined instead by the choice of $\mu$ with  $\mu^{\ell} = M$, which again corresponds to a factor $(L-1)$ of the generalized $A$-polynomial.
  The \emph{commutative} component of the generalized $A$-polynomial curve $\apolycurveg K$ is the one cut out by the factor $(L-1)$ of $\apolyg K$.

A knot $K$ is hyperbolic when its complement admits a discrete, faithful representation $\rho_{\text{hyp}}$ into $\operatorname{PSL}_2(\C)$, equivalently into $\SL_2(\C)$.
We call the component of $\charvar K$ (and the coresponding component of $\charvarg K$) the \emph{geometric} component, and similarly for their images in $\apolycurve K$ and $\apolycurveg K$.

\begin{thm}
  For a hyperbolic knot $K$ the restriction map $f : \charvarg K \to \apolycurveg K$ is a birational isomorphism on the geometric component.
\end{thm}
\begin{proof}
  The case of ordinary character varieties is a theorem \cite[Theorem 3.1]{MR1695208} of \citeauthor{MR1695208}.
  Since the extra data in the generalized case is exactly the same for the character variety and $A$-polynomial, $f$ is still an isomorphism.
\end{proof}

\begin{cor}
  For any hyperbolic knot $K$, $\kare$ is a rational function on the commutative and geometric components of the curve $\apolycurveg K$ cut out by $\apolyg K$.
\end{cor}
\begin{proof}
  The geometric and commutative components of $\apolycurveg K$ are in birational equivalence with those of $\charvarg K$, so it suffices to show that  $\kare$ is a rational function on $\charvarg K$.
  Given a point $P \in \charvarg K$ we can choose a colored diagram $D$ representing $K$ whose holonomy representation has character $P$.
  The $\Fcent[\Omega]$-characters $\chi_i$ (which we can think of as an algebraic set by using their values on $K^r, E^r, F^r$, and $\Omega$) coloring the strands of $D$ are rational functions of the point $P$.
  The value of $\kare$ is in turn determined by the matrix coefficients of the linear maps assigned to the diagram, which are rational functions of the $\chi_i$, which proves our claim.
\end{proof}

\section{Examples}
We compute some examples of $\kare(K)$ for $\ell = 3$, thinking of it as a function on the generalized $A$-polynomial curve.
{\scshape Mathematica} code giving these computations is given in \cite{KCCThesis}.

While we know abstractly that there are rational functions from the geometric component of $\apolyg K$ to that of $\repvarg K$, determining it explicitly in order to compute the invariant can be rather involved.
(In contrast, for the commutative component the correspondence is quite straightforward.)
For now we give a few simple examples computed by hand.
A better understanding of the relationship between colored diagrams and the $A$-polynomial could make this process more systematic.

\subsection{The trefoil knot}
\begin{figure}
  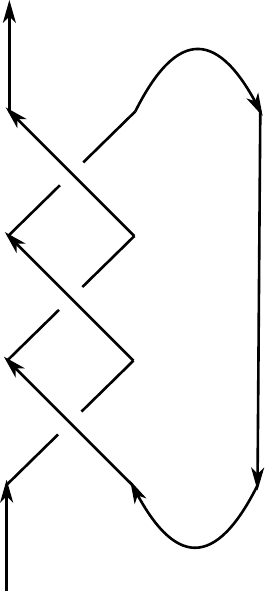
  \caption{A Morse presentation of the cut-open trefoil knot.
    The colors of the unlabeled segments are determined by the choice of $x_0$ and $y_0$.
  }
  \label{fig:trefoil-morse}
\end{figure}

In its $2$-bridge presentation the fundamental group of the trefoil knot $\trefoil$ is
\[
  \left\langle a, b, w \middle| w = ba, aw = bw \right\rangle.
\]
Eliminating $w$ gives the usual description of the trefoil knot group as the three strand braid group.  Since the trefoil is a torus knot, it also has a presentation $\left\langle r,s \middle| r^3 = s^2 \right\rangle$ where $r = ab$ and $s = aba$.

These correspond%
\footnote{
  Here we take a slightly different convention on the Wirtinger generators and thus conjugate by $x_0^-$ instead of $x_0^+$.
  This is to match the conventions of \cite{KCCThesis}.
}
to the colored diagram of Figure \ref{fig:trefoil-morse} as
\[
  a = x_0^+ (x_0^-)^{-1}, \quad
  b = x_0^- y_0^+ (y_0^-)^{-1} (x_0^-)^{-1}.
\]

In terms of these generators the meridian is $\mathfrak m = a$ and the longitude is 
\[
  \mathfrak l = a^{-4} w w^* = a^{-4}(ba)(ab)
\]
where $w^*$ is the reverse of the word $w$ in $a,b$.
The $A$-polynomial is $(L -1) (1 + LM^6)$, so the $1$-dimensional part of the character variety comes in two components, the commutative component $(L - 1) = 0$ and a noncommutative component $(1-LM^6)=0$.
In this case there is only one noncommutative component and we can work out an isomorphism with the $A$-polynomial curve by hand.

These two components come from the following representations of the knot group:
$$a \mapsto \left(\begin{array}{cc} M & 0 \\ 0 & M^{-1} \end{array}\right)$$
$$b \mapsto \left(\begin{array}{cc} M & 0 \\ 0 & M^{-1} \end{array}\right)$$
and
$$a \mapsto \left(\begin{array}{cc} M & 0 \\ 0 & M^{-1} \end{array}\right)$$
$$b \mapsto \left(\begin{array}{cc} \frac{-M^{-2}}{M-M^{-1}} & 1 \\ -1-\frac{1}{(M-M^{-1})^2} & \frac{M^{2}}{M-M^{-1}}\end{array}\right).$$

\begin{rem}
  The Thurston model geometry for a torus knot complement is an $\widetilde{\mathrm{SL}_2(\mathbb{R})}$-structure, not a $\mathrm{PSL}_2(\mathbb{C})$ (hyperbolic) structure,
and a map from the knot group to $\widetilde{\mathrm{SL}_2(\mathbb{R})}$ does not give a map to $\mathrm{SL}_2(\mathbb{C})$.
They can be compared in the sense that both give a projective two-dimensional representation, but there may be several lifts to $\mathrm{SL}_2(\mathbb{C})$ of the map from the knot group to $\mathrm{PSL}_2(\mathbb{R})$.

This corresponds to the fact that in general there is not a preferred geometric component of the $\mathrm{SL}_2(\mathbb{C})$ character variety of a torus knot group.
For a $(p,q)$ torus knot there are $(p-1)(q-1)/2$ nonabelian components \cite[Theorem 3.1]{MR2553945}, each of which will have a ``geometric'' point.
The trefoil knot (the case $p = 3, q = 2$) is somewhat special: there's only one noncommutative component, which has a preferred geometric point corresponding to the usual map from the trefoil knot group to $SL_2(\mathbb{Z})$.
\end{rem}

On the geometric component the point $(L=1, M=1)$ corresponds to the $\widetilde{\mathrm{SL}_2(\mathbb{R})}$-structure in the previous remark.
Explicitly this point is given by the following representation of the knot group (this formula is difficult to recover directly from the above formulas because they behave poorly as $M\rightarrow 1$):
$$a \mapsto \left(\begin{array}{cc} 1 & 1 \\ 0 & 1 \end{array}\right)$$
$$b \mapsto \left(\begin{array}{cc} 1 & 0 \\ -1 & 1\end{array}\right).$$
(Note that the ``torus knot" generators $r$ and $s$ get sent to the usual order $6$ and order $4$ elements in $\mathbb{SL}_2(\mathbb{Z})$.)

Set $\mu^3=M$, so that a representation of the generalized knot group $\pi^{(3)}(\trefoil)$ is determined by $\mu$ together with the above formulas.  We can now compute the Kashaev-Reshetikhin knot invariant at a $3$rd root of unity for the trefoil:
\begin{align*}
  \kare^3_{\operatorname{comm}}(\trefoil) & = \mu^8 + 2\mu^6+4\mu^4+5\mu^2+7+5\mu^{-2}+4\mu^{-4}+2\mu^{-6}+\mu^{-8}, \\
  \kare^3_{\operatorname{non-comm}}(\trefoil) & = 3(\mu^2+1+\mu^{-2})^2.
\end{align*}
In particular the value of $\kare^3(\trefoil)$ on the geometric point corresponds to $\mu$ a third root of unity.  If $\mu=1$ we get $27$, while if $\mu$ is a primitive cube root of unity we get $0$.

\subsection{The figure eight knot}
Now let $K = \figeight$ be the figure-eight knot.
In its $2$-bridge presentation the fundamental group of \figeight is
$$\langle a, b,w | w=ab^{-1}a^{-1}b, aw=wb\rangle.$$
As before, in terms of the colored diagram of Figure \ref{fig:figure-eight-morse} these are
\[
  a = x_0^+ (x_0^-)^{-1}, \quad
  b = x_0^- y_0^+ (y_0^-)^{-1} (x_0^-)^{-1}.
\]
In this presentation we can choose the meridian as $\mathfrak m = a$ and longitude as
\[
  \mathfrak l = ww^* = (ab^{-1}a^{-1}b)(ba^{-1}b^{-1}a).
\]
The $A$-polynomial is $(L-1)(L^2 M^4  + L (-M^8 + M^6 + 2M^4 + M^2 - 1) + M^4)$.
Thus the $1$-dimensional part of the character variety comes in two components, the commutative component $(L-1)=0$ and the geometric component
$$L^2 M^4 + L (-M^8 + M^6 + 2M^4 + M^2 - 1) + M^4=0.$$
These two components come from the following representations of the knot group:
$$a \mapsto \left(\begin{array}{cc} M & 0 \\ 0 & M^{-1} \end{array}\right)$$
$$b \mapsto \left(\begin{array}{cc} M & 0 \\ 0 & M^{-1} \end{array}\right)$$
and
$$a \mapsto \left(\begin{array}{cc} M & 0 \\ 0 & M^{-1} \end{array}\right)$$
$$b \mapsto \left(\begin{array}{cc} \frac{M^3(L+M^2)}{(M^2-1)^2(M^2+1)} & \frac{-1+M^2+(3+L)M^4+M^6-M^8}{M(M^2-1)^2(M^2+1)} \\ -\frac{M+LM^3}{(M^2-1)(M^2+1)} & \frac{1-(2+L)M^4-M^6+M^8}{M(M^2-1)^2(M^2+1)}\end{array}\right).$$

Explicitly the representation corresponding to the hyperbolic holonomy is:
$$a \mapsto \left(\begin{array}{cc} 1 & 1 \\ 0 & 1 \end{array}\right)$$
$$b \mapsto \left(\begin{array}{cc} 1 & 0 \\ -e^{\frac{2 \pi i}{3}} & 1 \end{array}\right).$$
Set $\mu^3=M$, so that a representation of the generalized knot group $\pi^{(3)}(K)$ is determined by $\mu$ together with the above formulas.
We can now compute the Kashaev-Reshetikhin knot invariant at a $3$rd root of unity as
\begin{align*}
  \kare^3_{\operatorname{geom}}(K)
  &= 3 \mu ^{8}+9 \mu ^{6}+21 \mu ^{4}+30 \mu ^{2}+36 +30 \mu ^{-2}+21 \mu ^{-4}+9 \mu ^{-6}+3\mu^{-8}
  \\
  &= 3 (\mu^2 + \mu^{-2})(\mu + 1 + \mu^{-1})^3(\mu - 1 + \mu^{-1})^3.
\end{align*}
In particular the value of $KaRe_3(K)$ on the hyperbolic holonomy gives $162$ for $\mu=1$ and $0$ for $\mu$ a primitive cube root of unity.

In the computation, we make use of the fact that the rational functions on the geometric component are a quadratic extension of $\C(\mu)$. This allows us to efficiently work with the (rather complicated!) intermediate expressions that arise as we build the figure eight knot out of its constituent crossings.

\subsection{Twist knots at the hyperbolic holonomy}
\begin{figure}
  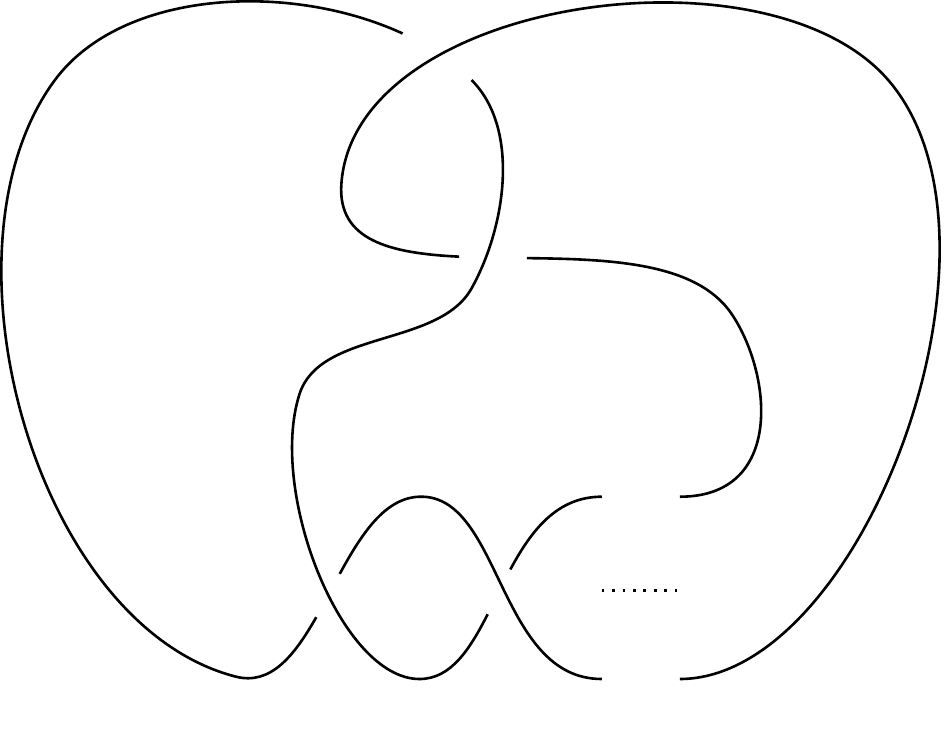
  \caption{The twist knot $T_m$ has $|m|$ crossings at the bottom, which are positive (as above) for $m > 0$ and negative otherwise.}
  \label{fig:twist-knot}
\end{figure}
Both of our previous examples are special cases of the \emph{twist knots} $T_m$.
These are a family of knots shown in Figure \ref{fig:twist-knot}; for $m = 0, 1, -1$ we obtain the unknot, the trefoil, and the figure-eight knot.
For $m \ne 0, 1$ the knot $T_n$ is hyperbolic.

\citeauthor{MR1831680} \cite{MR1831680} have carefully worked out the case where $m = -2n$ is even.
In that case fundamental group of a twist knot has the presentation
\[
  \pi_1(T_n) = 
  \left\langle a, b, w \middle| (b^{-1} a b a^{-1})^n, aw = wb \right\rangle
\]
and there is an explicit formula for the geometric point:
\begin{prop}
  Suppose that $m$ is even and not $0$ or $2$, and set $n = - m/2$.
  Then the hyperbolic point of the twist knot $\knotgrp(T_m)$ is given by
  \begin{align*}
    a&=\left(\begin{array}{cc}1 & 1 \\ 0 & 1\end{array}\right) \\
    b&=\left(\begin{array}{cc}1 & 0 \\ -(x-x^{-1}) & 1\end{array}\right)
  \end{align*}
  where $x$ depends on the sign of $n$:
  \begin{enumerate}
    \item If $n$ is positive, $x$ is the unique root of 
      \[
        \frac{x^{4n+2} + x^{4n+1} - x + 1}{x^2 +1}
        \text{ with }
        \frac{(2n - 2) \pi}{4n}
        <
        \arg(x)
        <
        \frac{(2n - 1) \pi}{4n}.
      \]
    \item If $n$ is negative, $x$ is the unique root of 
      \[
        \frac{-x^{-4n} + x^{-4n-1} + x + 1}{x^2 +1}
        \text{ with }
        \frac{(2n + 3) \pi}{4n}
        <
        \arg(x)
        <
        \frac{(2n + 2) \pi}{4n}.
      \]
  \end{enumerate}
\end{prop}
\begin{proof}
  \cite[Section 3]{MR1831680}.
\end{proof}

\begin{rem}
  Choosing the correct root of $x$ is important here.
  The hyperbolic point of the character variety $\charvar K$ is a singular point of the $A$-polynomial curve $\apolycurve K$.
  Consequently every root of the above polynomial corresponds to the same singular point but to different virtual points on its smooth cover.
  Because these come from different components of $\charvar K$, hence $\charvarg K$, the invariant $\kare$ will in general take different values on different virtual points.
\end{rem}

The following table gives the values of $\kare(T_{-2n})$ evaluated at the hyperbolic point for both $\mu = 1$ and $\mu = \exp(2\pi i/3)$ a primitive third root of unity.
\[
  \begin{array}{c|c|c}
    n &KaRe_3(K_{-2n})(\rho_\text{hyp},1)&KaRe_3(K_{-2n})(\rho_{\text{hyp}}, \exp(2\pi i/3))\\
    \hline
    1 & 162&0 \\
    2 & 168.986 - 54.5257 i&0\\
    3 &418.375+ 30.4018 i&0 \\
    4 & 352.139+ 47.9167 i&0\\
    5 & 673.376- 37.3665 i&0 \\
    6 & 1112.4+ 63.914 i&0\\
    7 &1059.55+ 75.2 i&0 \\
  \end{array}
\]
We do not know any reason that the second column is uniformly zero.
The pattern appears to persist for $\ell = 5$.

\printbibliography

\end{document}

%% file: preamble.tex
\usepackage{amsmath, amssymb, amsthm, mathtools}

\usepackage{tikz-cd}

\usepackage{xparse}

\usepackage{xcolor}

\usepackage{microtype}

\usepackage{graphicx,subcaption}
\graphicspath{{fig/}}

\newcommand{\C}{\mathbb{C}}
\newcommand{\Q}{\mathbb{Q}}
\newcommand{\Z}{\mathbb{Z}}
\newcommand{\SL}{\mathrm{SL}}

\newcommand{\End}[1]{\operatorname{End}\left(#1\right)}
\newcommand{\Hom}[3]{\operatorname{Hom}_{#1}\left(#2,#3\right)}

\DeclareMathOperator{\cb}{Cb}
\DeclareMathOperator{\kare}{KR}
\DeclareMathOperator{\id}{id}
\DeclareMathOperator{\spec}{Spec}
\DeclareMathOperator{\tr}{tr}

\DeclareDocumentCommand \funcalg { o } {%
  \IfNoValueTF {#1} {%
    \operatorname{F}_{\zeta}(\SL_2^*(\mathbb{C}))%
  }{%
    \operatorname{F}_{#1}(\SL_2^*(\mathbb{C}))%
  }%
}

\newcommand{\Fcent}{\mathcal{Z}_0}
\newcommand{\modcat}{\mathcal{C}}
\newcommand{\irrep}[1]{V_{#1}}
\newcommand{\rmat}{\mathcal{R}}

\newcommand{\repvar}[1]{\mathfrak{R}_{#1}}
\newcommand{\charvar}[1]{\mathfrak{X}_{#1}}
\newcommand{\apolycurve}[1]{\mathfrak{A}_{#1}}
\DeclareDocumentCommand{\repvarg}{m O{\ell} }{\mathfrak{R}^{(#2)}_{#1}}
\DeclareDocumentCommand{\charvarg}{m O{\ell} }{\mathfrak{X}^{(#2)}_{#1}}
\DeclareDocumentCommand{\apolycurveg}{m O{\ell} }{\mathfrak{A}^{(#2)}_{#1}}
\newcommand{\apoly}[1]{A_{#1}}
\DeclareDocumentCommand{\apolyg}{m O{\ell} }{A^{(#2)}_{#1}}

\newcommand{\knotgrp}{\pi_1}
\DeclareDocumentCommand{\knotgrpg}{O{\ell}}{\pi_1^{(#1)}}

\newcommand{\defeq}{\mathrel{:=}}
\newcommand{\iso}{\cong}

\newcommand{\trefoil}{\ensuremath{3_1}}
\newcommand{\figeight}{\ensuremath{4_1}}

\newtheorem{prop}{Proposition}[section]
\newtheorem{thm}[prop]{Theorem}
\newtheorem{lem}[prop]{Lemma}
\newtheorem{cor}[prop]{Corollary}
\theoremstyle{definition}
\newtheorem{defn}[prop]{Definition}
\newtheorem{rem}[prop]{Remark}

\usepackage[style=numeric,maxnames=6]{biblatex}

%% file: fig/crossing-rule-positive.pdf_tex
\begingroup%
  \makeatletter%
  \providecommand\color[2][]{%
    \errmessage{(Inkscape) Color is used for the text in Inkscape, but the package 'color.sty' is not loaded}%
    \renewcommand\color[2][]{}%
  }%
  \providecommand\transparent[1]{%
    \errmessage{(Inkscape) Transparency is used (non-zero) for the text in Inkscape, but the package 'transparent.sty' is not loaded}%
    \renewcommand\transparent[1]{}%
  }%
  \providecommand\rotatebox[2]{#2}%
  \newcommand*\fsize{\dimexpr\f@size pt\relax}%
  \newcommand*\lineheight[1]{\fontsize{\fsize}{#1\fsize}\selectfont}%
  \ifx\svgwidth\undefined%
    \setlength{\unitlength}{105.78839207bp}%
    \ifx\svgscale\undefined%
      \relax%
    \else%
      \setlength{\unitlength}{\unitlength * \real{\svgscale}}%
    \fi%
  \else%
    \setlength{\unitlength}{\svgwidth}%
  \fi%
  \global\let\svgwidth\undefined%
  \global\let\svgscale\undefined%
  \makeatother%
  \begin{picture}(1,0.82822064)%
    \lineheight{1}%
    \setlength\tabcolsep{0pt}%
    \put(0,0){\includegraphics[width=\unitlength,page=1]{crossing-rule-positive.pdf}}%
    \put(0.03366009,0.00347595){\color[rgb]{0,0,0}\makebox(0,0)[lt]{\lineheight{1.25}\smash{\begin{tabular}[t]{l}$\chi_1$\end{tabular}}}}%
    \put(0.88753004,0.00471369){\color[rgb]{0,0,0}\makebox(0,0)[lt]{\lineheight{1.25}\smash{\begin{tabular}[t]{l}$\chi_2$\end{tabular}}}}%
    \put(0.03397065,0.79645974){\color[rgb]{0,0,0}\makebox(0,0)[lt]{\lineheight{1.25}\smash{\begin{tabular}[t]{l}$\chi_4$\end{tabular}}}}%
    \put(0.88784057,0.79769746){\color[rgb]{0,0,0}\makebox(0,0)[lt]{\lineheight{1.25}\smash{\begin{tabular}[t]{l}$\chi_3$\end{tabular}}}}%
  \end{picture}%
\endgroup%

%% file: fig/crossing-rule-negative.pdf_tex
\begingroup%
  \makeatletter%
  \providecommand\color[2][]{%
    \errmessage{(Inkscape) Color is used for the text in Inkscape, but the package 'color.sty' is not loaded}%
    \renewcommand\color[2][]{}%
  }%
  \providecommand\transparent[1]{%
    \errmessage{(Inkscape) Transparency is used (non-zero) for the text in Inkscape, but the package 'transparent.sty' is not loaded}%
    \renewcommand\transparent[1]{}%
  }%
  \providecommand\rotatebox[2]{#2}%
  \newcommand*\fsize{\dimexpr\f@size pt\relax}%
  \newcommand*\lineheight[1]{\fontsize{\fsize}{#1\fsize}\selectfont}%
  \ifx\svgwidth\undefined%
    \setlength{\unitlength}{105.78839207bp}%
    \ifx\svgscale\undefined%
      \relax%
    \else%
      \setlength{\unitlength}{\unitlength * \real{\svgscale}}%
    \fi%
  \else%
    \setlength{\unitlength}{\svgwidth}%
  \fi%
  \global\let\svgwidth\undefined%
  \global\let\svgscale\undefined%
  \makeatother%
  \begin{picture}(1,0.82822064)%
    \lineheight{1}%
    \setlength\tabcolsep{0pt}%
    \put(0,0){\includegraphics[width=\unitlength,page=1]{crossing-rule-negative.pdf}}%
    \put(0.03366009,0.00347595){\color[rgb]{0,0,0}\makebox(0,0)[lt]{\lineheight{1.25}\smash{\begin{tabular}[t]{l}$\chi_1$\end{tabular}}}}%
    \put(0.88753004,0.00471369){\color[rgb]{0,0,0}\makebox(0,0)[lt]{\lineheight{1.25}\smash{\begin{tabular}[t]{l}$\chi_2$\end{tabular}}}}%
    \put(0.03397065,0.79645974){\color[rgb]{0,0,0}\makebox(0,0)[lt]{\lineheight{1.25}\smash{\begin{tabular}[t]{l}$\chi_4$\end{tabular}}}}%
    \put(0.88784057,0.79769746){\color[rgb]{0,0,0}\makebox(0,0)[lt]{\lineheight{1.25}\smash{\begin{tabular}[t]{l}$\chi_3$\end{tabular}}}}%
  \end{picture}%
\endgroup%

%% file: fig/figure-eight.pdf_tex
\begingroup%
  \makeatletter%
  \providecommand\color[2][]{%
    \errmessage{(Inkscape) Color is used for the text in Inkscape, but the package 'color.sty' is not loaded}%
    \renewcommand\color[2][]{}%
  }%
  \providecommand\transparent[1]{%
    \errmessage{(Inkscape) Transparency is used (non-zero) for the text in Inkscape, but the package 'transparent.sty' is not loaded}%
    \renewcommand\transparent[1]{}%
  }%
  \providecommand\rotatebox[2]{#2}%
  \newcommand*\fsize{\dimexpr\f@size pt\relax}%
  \newcommand*\lineheight[1]{\fontsize{\fsize}{#1\fsize}\selectfont}%
  \ifx\svgwidth\undefined%
    \setlength{\unitlength}{288bp}%
    \ifx\svgscale\undefined%
      \relax%
    \else%
      \setlength{\unitlength}{\unitlength * \real{\svgscale}}%
    \fi%
  \else%
    \setlength{\unitlength}{\svgwidth}%
  \fi%
  \global\let\svgwidth\undefined%
  \global\let\svgscale\undefined%
  \makeatother%
  \begin{picture}(1,1)%
    \lineheight{1}%
    \setlength\tabcolsep{0pt}%
    \put(0,0){\includegraphics[width=\unitlength,page=1]{figure-eight.pdf}}%
    \put(0.07812503,0.0625){\color[rgb]{0,0,0}\makebox(0,0)[lt]{\lineheight{1.25}\smash{\begin{tabular}[t]{l}$x_0$\end{tabular}}}}%
    \put(0.61979159,0.12499998){\color[rgb]{0,0,0}\makebox(0,0)[lt]{\lineheight{1.25}\smash{\begin{tabular}[t]{l}$y_0$\end{tabular}}}}%
    \put(0.01298716,0.375){\color[rgb]{0,0,0}\makebox(0,0)[lt]{\lineheight{1.25}\smash{\begin{tabular}[t]{l}$x_1$\end{tabular}}}}%
    \put(0.00641977,0.62709096){\color[rgb]{0,0,0}\makebox(0,0)[lt]{\lineheight{1.25}\smash{\begin{tabular}[t]{l}$x_2$\end{tabular}}}}%
    \put(0.87500002,0.85){\color[rgb]{0,0,0}\makebox(0,0)[lt]{\lineheight{1.25}\smash{\begin{tabular}[t]{l}$x_3$\end{tabular}}}}%
  \end{picture}%
\endgroup%

%% file: fig/figure-eight-path-examples.pdf_tex
\begingroup%
  \makeatletter%
  \providecommand\color[2][]{%
    \errmessage{(Inkscape) Color is used for the text in Inkscape, but the package 'color.sty' is not loaded}%
    \renewcommand\color[2][]{}%
  }%
  \providecommand\transparent[1]{%
    \errmessage{(Inkscape) Transparency is used (non-zero) for the text in Inkscape, but the package 'transparent.sty' is not loaded}%
    \renewcommand\transparent[1]{}%
  }%
  \providecommand\rotatebox[2]{#2}%
  \newcommand*\fsize{\dimexpr\f@size pt\relax}%
  \newcommand*\lineheight[1]{\fontsize{\fsize}{#1\fsize}\selectfont}%
  \ifx\svgwidth\undefined%
    \setlength{\unitlength}{288bp}%
    \ifx\svgscale\undefined%
      \relax%
    \else%
      \setlength{\unitlength}{\unitlength * \real{\svgscale}}%
    \fi%
  \else%
    \setlength{\unitlength}{\svgwidth}%
  \fi%
  \global\let\svgwidth\undefined%
  \global\let\svgscale\undefined%
  \makeatother%
  \begin{picture}(1,1)%
    \lineheight{1}%
    \setlength\tabcolsep{0pt}%
    \put(0,0){\includegraphics[width=\unitlength,page=1]{figure-eight-path-examples.pdf}}%
    \put(0.07812503,0.0625){\color[rgb]{0,0,0}\makebox(0,0)[lt]{\lineheight{1.25}\smash{\begin{tabular}[t]{l}$x_0$\end{tabular}}}}%
    \put(0.61979159,0.12499998){\color[rgb]{0,0,0}\makebox(0,0)[lt]{\lineheight{1.25}\smash{\begin{tabular}[t]{l}$y_0$\end{tabular}}}}%
    \put(0.01298716,0.375){\color[rgb]{0,0,0}\makebox(0,0)[lt]{\lineheight{1.25}\smash{\begin{tabular}[t]{l}$x_1$\end{tabular}}}}%
    \put(0.00641977,0.62709096){\color[rgb]{0,0,0}\makebox(0,0)[lt]{\lineheight{1.25}\smash{\begin{tabular}[t]{l}$x_2$\end{tabular}}}}%
    \put(0.87500002,0.85){\color[rgb]{0,0,0}\makebox(0,0)[lt]{\lineheight{1.25}\smash{\begin{tabular}[t]{l}$x_3$\end{tabular}}}}%
    \put(0,0){\includegraphics[width=\unitlength,page=2]{figure-eight-path-examples.pdf}}%
  \end{picture}%
\endgroup%

%% file: fig/groupoid-relations.pdf_tex
\begingroup%
  \makeatletter%
  \providecommand\color[2][]{%
    \errmessage{(Inkscape) Color is used for the text in Inkscape, but the package 'color.sty' is not loaded}%
    \renewcommand\color[2][]{}%
  }%
  \providecommand\transparent[1]{%
    \errmessage{(Inkscape) Transparency is used (non-zero) for the text in Inkscape, but the package 'transparent.sty' is not loaded}%
    \renewcommand\transparent[1]{}%
  }%
  \providecommand\rotatebox[2]{#2}%
  \newcommand*\fsize{\dimexpr\f@size pt\relax}%
  \newcommand*\lineheight[1]{\fontsize{\fsize}{#1\fsize}\selectfont}%
  \ifx\svgwidth\undefined%
    \setlength{\unitlength}{105.78839207bp}%
    \ifx\svgscale\undefined%
      \relax%
    \else%
      \setlength{\unitlength}{\unitlength * \real{\svgscale}}%
    \fi%
  \else%
    \setlength{\unitlength}{\svgwidth}%
  \fi%
  \global\let\svgwidth\undefined%
  \global\let\svgscale\undefined%
  \makeatother%
  \begin{picture}(1,0.82822064)%
    \lineheight{1}%
    \setlength\tabcolsep{0pt}%
    \put(0,0){\includegraphics[width=\unitlength,page=1]{groupoid-relations.pdf}}%
    \put(0.03366009,0.00347595){\color[rgb]{0,0,0}\makebox(0,0)[lt]{\lineheight{1.25}\smash{\begin{tabular}[t]{l}$\chi_1$\end{tabular}}}}%
    \put(0.88753004,0.00471369){\color[rgb]{0,0,0}\makebox(0,0)[lt]{\lineheight{1.25}\smash{\begin{tabular}[t]{l}$\chi_2$\end{tabular}}}}%
    \put(0.03397065,0.79645974){\color[rgb]{0,0,0}\makebox(0,0)[lt]{\lineheight{1.25}\smash{\begin{tabular}[t]{l}$\chi_4$\end{tabular}}}}%
    \put(0.88784057,0.79769746){\color[rgb]{0,0,0}\makebox(0,0)[lt]{\lineheight{1.25}\smash{\begin{tabular}[t]{l}$\chi_3$\end{tabular}}}}%
    \put(0,0){\includegraphics[width=\unitlength,page=2]{groupoid-relations.pdf}}%
  \end{picture}%
\endgroup%

%% file: fig/figure-eight-morse.pdf_tex
\begingroup%
  \makeatletter%
  \providecommand\color[2][]{%
    \errmessage{(Inkscape) Color is used for the text in Inkscape, but the package 'color.sty' is not loaded}%
    \renewcommand\color[2][]{}%
  }%
  \providecommand\transparent[1]{%
    \errmessage{(Inkscape) Transparency is used (non-zero) for the text in Inkscape, but the package 'transparent.sty' is not loaded}%
    \renewcommand\transparent[1]{}%
  }%
  \providecommand\rotatebox[2]{#2}%
  \newcommand*\fsize{\dimexpr\f@size pt\relax}%
  \newcommand*\lineheight[1]{\fontsize{\fsize}{#1\fsize}\selectfont}%
  \ifx\svgwidth\undefined%
    \setlength{\unitlength}{92.92604542bp}%
    \ifx\svgscale\undefined%
      \relax%
    \else%
      \setlength{\unitlength}{\unitlength * \real{\svgscale}}%
    \fi%
  \else%
    \setlength{\unitlength}{\svgwidth}%
  \fi%
  \global\let\svgwidth\undefined%
  \global\let\svgscale\undefined%
  \makeatother%
  \begin{picture}(1,3.10329428)%
    \lineheight{1}%
    \setlength\tabcolsep{0pt}%
    \put(0,0){\includegraphics[width=\unitlength,page=1]{figure-eight-morse.pdf}}%
    \put(0.03799147,0.12341339){\makebox(0,0)[lt]{\lineheight{1.25}\smash{\begin{tabular}[t]{l}$x_0$\end{tabular}}}}%
    \put(0.684839,0.18619891){\makebox(0,0)[lt]{\lineheight{1.25}\smash{\begin{tabular}[t]{l}$y_0$\end{tabular}}}}%
  \end{picture}%
\endgroup%

%% file: fig/gauge-transf-example.pdf_tex
\begingroup%
  \makeatletter%
  \providecommand\color[2][]{%
    \errmessage{(Inkscape) Color is used for the text in Inkscape, but the package 'color.sty' is not loaded}%
    \renewcommand\color[2][]{}%
  }%
  \providecommand\transparent[1]{%
    \errmessage{(Inkscape) Transparency is used (non-zero) for the text in Inkscape, but the package 'transparent.sty' is not loaded}%
    \renewcommand\transparent[1]{}%
  }%
  \providecommand\rotatebox[2]{#2}%
  \newcommand*\fsize{\dimexpr\f@size pt\relax}%
  \newcommand*\lineheight[1]{\fontsize{\fsize}{#1\fsize}\selectfont}%
  \ifx\svgwidth\undefined%
    \setlength{\unitlength}{136.25160027bp}%
    \ifx\svgscale\undefined%
      \relax%
    \else%
      \setlength{\unitlength}{\unitlength * \real{\svgscale}}%
    \fi%
  \else%
    \setlength{\unitlength}{\svgwidth}%
  \fi%
  \global\let\svgwidth\undefined%
  \global\let\svgscale\undefined%
  \makeatother%
  \begin{picture}(1,0.98419008)%
    \lineheight{1}%
    \setlength\tabcolsep{0pt}%
    \put(0,0){\includegraphics[width=\unitlength,page=1]{gauge-transf-example.pdf}}%
    \put(0.78763056,0.46592908){\makebox(0,0)[lt]{\lineheight{1.25}\smash{\begin{tabular}[t]{l}$D'$\end{tabular}}}}%
    \put(0,0){\includegraphics[width=\unitlength,page=2]{gauge-transf-example.pdf}}%
    \put(0.10285914,0.45880857){\makebox(0,0)[lt]{\lineheight{1.25}\smash{\begin{tabular}[t]{l}$D$\end{tabular}}}}%
    \put(0,0){\includegraphics[width=\unitlength,page=3]{gauge-transf-example.pdf}}%
    \put(0.10885901,0.00507418){\makebox(0,0)[lt]{\lineheight{1.25}\smash{\begin{tabular}[t]{l}$\chi$\end{tabular}}}}%
    \put(0.10885887,0.95470169){\makebox(0,0)[lt]{\lineheight{1.25}\smash{\begin{tabular}[t]{l}$\chi$\end{tabular}}}}%
    \put(0.26066667,0.0058221){\makebox(0,0)[lt]{\lineheight{1.25}\smash{\begin{tabular}[t]{l}$\gamma$\end{tabular}}}}%
    \put(0.26066651,0.95544961){\makebox(0,0)[lt]{\lineheight{1.25}\smash{\begin{tabular}[t]{l}$\gamma$\end{tabular}}}}%
    \put(0.80195207,0.00438896){\makebox(0,0)[lt]{\lineheight{1.25}\smash{\begin{tabular}[t]{l}$\chi$\end{tabular}}}}%
    \put(0.80195191,0.95401647){\makebox(0,0)[lt]{\lineheight{1.25}\smash{\begin{tabular}[t]{l}$\chi$\end{tabular}}}}%
    \put(0.95375994,0.00513688){\makebox(0,0)[lt]{\lineheight{1.25}\smash{\begin{tabular}[t]{l}$\gamma$\end{tabular}}}}%
    \put(0.95375978,0.95476439){\makebox(0,0)[lt]{\lineheight{1.25}\smash{\begin{tabular}[t]{l}$\gamma$\end{tabular}}}}%
    \put(0.45645964,0.46006465){\color[rgb]{0,0,0}\makebox(0,0)[lt]{\lineheight{1.25}\smash{\begin{tabular}[t]{l}$\Rightarrow$\end{tabular}}}}%
    \put(0.85367545,0.27352925){\color[rgb]{0,0,0}\makebox(0,0)[lt]{\lineheight{1.25}\smash{\begin{tabular}[t]{l}$\chi'$\end{tabular}}}}%
    \put(0.85367545,0.63132324){\color[rgb]{0,0,0}\makebox(0,0)[lt]{\lineheight{1.25}\smash{\begin{tabular}[t]{l}$\chi'$\end{tabular}}}}%
  \end{picture}%
\endgroup%

%% file: fig/trefoil-morse.pdf_tex
\begingroup%
  \makeatletter%
  \providecommand\color[2][]{%
    \errmessage{(Inkscape) Color is used for the text in Inkscape, but the package 'color.sty' is not loaded}%
    \renewcommand\color[2][]{}%
  }%
  \providecommand\transparent[1]{%
    \errmessage{(Inkscape) Transparency is used (non-zero) for the text in Inkscape, but the package 'transparent.sty' is not loaded}%
    \renewcommand\transparent[1]{}%
  }%
  \providecommand\rotatebox[2]{#2}%
  \newcommand*\fsize{\dimexpr\f@size pt\relax}%
  \newcommand*\lineheight[1]{\fontsize{\fsize}{#1\fsize}\selectfont}%
  \ifx\svgwidth\undefined%
    \setlength{\unitlength}{76.11516953bp}%
    \ifx\svgscale\undefined%
      \relax%
    \else%
      \setlength{\unitlength}{\unitlength * \real{\svgscale}}%
    \fi%
  \else%
    \setlength{\unitlength}{\svgwidth}%
  \fi%
  \global\let\svgwidth\undefined%
  \global\let\svgscale\undefined%
  \makeatother%
  \begin{picture}(1,2.23548442)%
    \lineheight{1}%
    \setlength\tabcolsep{0pt}%
    \put(0,0){\includegraphics[width=\unitlength,page=1]{trefoil-morse.pdf}}%
    \put(0.0615566,0.04067406){\makebox(0,0)[lt]{\lineheight{1.25}\smash{\begin{tabular}[t]{l}$x_0$\end{tabular}}}}%
    \put(0.87194979,0.12085479){\makebox(0,0)[lt]{\lineheight{1.25}\smash{\begin{tabular}[t]{l}$y_0$\end{tabular}}}}%
    \put(0.76760633,2.11981775){\makebox(0,0)[lt]{\lineheight{1.25}\smash{\begin{tabular}[t]{l}$y_0$\end{tabular}}}}%
    \put(0.06154538,2.08057516){\makebox(0,0)[lt]{\lineheight{1.25}\smash{\begin{tabular}[t]{l}$x_0$\end{tabular}}}}%
  \end{picture}%
\endgroup%

%% file: fig/twist-knot.pdf_tex
\begingroup%
  \makeatletter%
  \providecommand\color[2][]{%
    \errmessage{(Inkscape) Color is used for the text in Inkscape, but the package 'color.sty' is not loaded}%
    \renewcommand\color[2][]{}%
  }%
  \providecommand\transparent[1]{%
    \errmessage{(Inkscape) Transparency is used (non-zero) for the text in Inkscape, but the package 'transparent.sty' is not loaded}%
    \renewcommand\transparent[1]{}%
  }%
  \providecommand\rotatebox[2]{#2}%
  \newcommand*\fsize{\dimexpr\f@size pt\relax}%
  \newcommand*\lineheight[1]{\fontsize{\fsize}{#1\fsize}\selectfont}%
  \ifx\svgwidth\undefined%
    \setlength{\unitlength}{271.87600136bp}%
    \ifx\svgscale\undefined%
      \relax%
    \else%
      \setlength{\unitlength}{\unitlength * \real{\svgscale}}%
    \fi%
  \else%
    \setlength{\unitlength}{\svgwidth}%
  \fi%
  \global\let\svgwidth\undefined%
  \global\let\svgscale\undefined%
  \makeatother%
  \begin{picture}(1,0.78362519)%
    \lineheight{1}%
    \setlength\tabcolsep{0pt}%
    \put(0,0){\includegraphics[width=\unitlength,page=1]{twist-knot.pdf}}%
    \put(0.3891805,0.00918098){\makebox(0,0)[lt]{\lineheight{1.25}\smash{\begin{tabular}[t]{l}$|m|$ crossings\end{tabular}}}}%
  \end{picture}%
\endgroup%